\documentclass[hidelinks,onefignum,onetabnum]{siamart251216}

\usepackage{lipsum}
\usepackage{amsfonts}
\usepackage{graphicx}
\usepackage{epstopdf}
\usepackage{algorithmic}
\ifpdf
  \DeclareGraphicsExtensions{.eps,.pdf,.png,.jpg}
\else
  \DeclareGraphicsExtensions{.eps}
\fi

\newsiamremark{remark}{Remark}
\newsiamremark{hypothesis}{Hypothesis}
\crefname{hypothesis}{Hypothesis}{Hypotheses}
\newsiamthm{claim}{Claim}
\newsiamremark{fact}{Fact}
\crefname{fact}{Fact}{Facts}

\title{Dynamic subscales and unconditional semi-discrete stability for non-residual VMS approximations of generalised Newtonian flows\thanks{Corresponding author: \texttt{ernesto.castillode@usach.cl}}}

\author{Gabriel R. Barrenechea\thanks{Department of Mathematics and Statistics University of Strathclyde, 26 Richmond Street, Glasgow G1 1XH, United Kingdom
		(\texttt{gabriel.barrenechea@strath.ac.uk}).}
	\and Ernesto Castillo\thanks{Department of Mechanical Engineering; Computational Heat and Fluid Flow Laboratory, Universidad de Santiago de Chile, Av. Libertador Bernardo O'Higgins 3363, Santiago 9170022, Chile
		(\texttt{ernesto.castillode@usach.cl}).}
	\and Ramon Codina\thanks{Department of Civil and Environmental Engineering Universitat Politècnica de Catalunya, Jordi Girona 1--3, Barcelona 08034, Spain
		(\texttt{ramon.codina@upc.edu}).}
	\and \mbox{Jorge Pastene}\thanks{Department of Mathematics and Computer Science; Computational Heat and Fluid Flow Laboratory, Universidad de Santiago de Chile, Av. Libertador Bernardo O'Higgins 3363, Santiago 9170022, Chile
		(\texttt{jorge.pastene.g@usach.cl}).}
	\and Patrick Vega\thanks{Department of Mathematics and Computer Science; Computational Heat and Fluid Flow Laboratory, Universidad de Santiago de Chile, Av. Libertador Bernardo O'Higgins 3363, Santiago 9170022, Chile
		(\texttt{patrick.vega@usach.cl}).}}

\usepackage{amsopn}

\newtheorem{assumption}[theorem]{Assumption}

\usepackage{amssymb}
\usepackage{mathrsfs}

\newcommand{\triple}[1]{\left|\!\left|\!\left|{#1}\right|\!\right|\!\right|}
\newcommand{\double}[1]{\left|\!\left|{#1}\right|\!\right|}

\newcommand{\grad}{\boldsymbol \nabla}
\renewcommand{\div}{\textrm{div}\,}

\newcommand{\BU}{\boldsymbol U}
\newcommand{\BV}{\boldsymbol V}

\newcommand{\bff}{\boldsymbol f}

\newcommand{\bu}{\boldsymbol u}
\newcommand{\bv}{\boldsymbol v}
\newcommand{\bw}{\boldsymbol w}
\newcommand{\bx}{\boldsymbol x}

\newcommand{\bz}{\boldsymbol z}

\newcommand{\LB}{\mathscr B}

\newcommand{\LP}{\mathscr P}

\newcommand{\LU}{\mathscr U}
\newcommand{\LV}{\mathscr V}
\newcommand{\LW}{\mathscr W}
\newcommand{\LX}{\mathscr X}

\def\ddx{{\,\rm d}x}
\def\dds{{\,\rm d}s}
\def\Sh{S_h}

\ifpdf
\hypersetup{
  pdftitle={},
  pdfauthor={G. R. Barrenechea, E. Castullo, R. Codina, J. Pastene, and P. Vega}
}
\fi

\usepackage{pdflscape}
\usepackage{amsmath}
\usepackage{multirow}
\usepackage{booktabs}

\begin{document}

\maketitle

\begin{abstract}
We analyse a non-residual variational multiscale finite element formulation with dynamic subscales for incompressible generalised Newtonian Navier--Stokes flows. The apparent viscosity is assumed to be bounded and Lipschitz continuous with respect to the shear rate, covering several regularised rheological laws. The method separates, through orthogonal projections, the unresolved pressure-divergence, convective, and pressure-gradient contributions. For the linearised semi-discrete problem, we prove well-posedness and an unconditional stability estimate in an anisotropic VMS norm. The estimate controls viscous dissipation, subscale energies, discrete divergence, and a coupled acceleration-convection-pressure balance. For the non-linear formulation, a fixed-point argument yields existence and optimal-order a priori bounds under suitable regularity and smallness assumptions. The analysis also provides weak-in-time control of the discrete pressure-convection balance. A key point is that the dynamic pressure-related subscale provides the time-derivative contribution needed to control the projected pressure-gradient error, a mechanism unavailable in the corresponding quasi-static setting.
\end{abstract}



\section{Introduction}

The numerical approximation of incompressible flows with non-linear viscous response remains a challenging problem in the analysis of finite element methods for fluid mechanics. Generalised Newtonian models extend the classical Navier--Stokes equations by introducing an apparent viscosity depending on the local shear rate. This leads to a system in which non-linear diffusion, incompressibility and convective transport are coupled. Such models arise in biofluid mechanics \cite{GaldiRannacherRobertsonTurek2008}, polymeric liquids \cite{BirdArmstrongHassager1987}, suspensions and other applications involving complex fluids \cite{Larson1999}, but they also introduce analytical difficulties related to the non-linear viscous operator, pressure control, convective stability and the coupling between resolved and unresolved scales.

The mathematical and numerical analysis of generalised Newtonian and implicitly constituted incompressible fluids has developed substantially in recent years. For steady flows, finite element convergence has been studied for implicit power-law-like rheologies using compactness and Lipschitz truncation techniques \cite{DieningKreuzerSuli2013}, and stabilised finite element methods have been analysed for power-law fluids \cite{BarrenecheaSuli2023}. Iterative linearisation strategies have also been investigated, including convergence-rate results for Ka\v{c}anov-type schemes applied to shear-thinning models \cite{PascalSuli2022}. For unsteady problems, fully discrete mixed finite element approximations have been developed for implicitly constituted incompressible fluids \cite{SuliTabea2019}, while mixed three-field formulations have been analysed for implicit constitutive models in which the stress is treated as an additional unknown \cite{Farrell2020}. 
Optimal error estimates for space-time discretisations of incompressible generalised Newtonian fluids have also been obtained under suitable structural assumptions on the extra stress tensor \cite{BerselliRusicka2021}. These works provide a rigorous analytical foundation for non-linear rheological models. 

Beyond this algebraic framework of shear-dependent viscosity laws, related non-residual VMS methodology has also been developed for viscoelastic fluids, a distinct class of non-Newtonian models in which the extra stress obeys its own evolution equation rather than an algebraic shear-dependent law. Stabilised finite element formulations for the three-field incompressible viscoelastic flow problem were proposed in \cite{Castillo2014}. Their analysis was subsequently carried out for a time-dependent semi-discrete scheme in \cite{BarrenecheaCastillo2019}, while the steady-state and linearised Navier--Stokes/Oldroyd-B problem was analysed in \cite{Castillo2017}.

Returning to the generalised Newtonian setting, the present work focuses on flows with bounded shear-dependent apparent viscosity. This class includes Carreau, Carreau--Yasuda, Cross, Quemada and Eyring-type laws \cite{Carreau1972,CarreauYasuda,Cross1965,Quemada1977,Eyring1936}, for which the viscosity is regularised between two strictly positive limiting values. From the analytical viewpoint, these assumptions provide a uniformly elliptic viscous contribution and a Lipschitz dependence on the shear rate, which are used to control the non-linear viscosity terms in the stability and error estimates. From the modelling viewpoint, they retain the structure of physically motivated apparent-viscosity laws, rather than replacing the rheology by an abstract non-linear diffusion operator. Thus, the analysis is developed for a mathematically tractable class of models while remaining connected with rheological laws used in applications.

These works, however, do not address the stability mechanism generated by dynamic non-residual variational multiscale subscales. A second difficulty is associated with the stabilisation of incompressible flows in convection-dominated regimes. Inf-sup stability controls the pressure-velocity coupling, but it does not by itself provide sufficient control of unresolved velocity scales generated by convection. Stabilised finite element methods address this issue by incorporating the effect of subgrid scales on the resolved finite element solution. This idea is rooted in the variational multiscale framework \cite{Hughes1995,Hughes1998} and has led to residual-based, SUPG-type, orthogonal subscale and projection-based stabilisation methods \cite{BrooksHughes1982,Codina1997,Codina2008,Badia2010}. In convection-dominated and under-resolved regimes, stabilisation provides a discrete mechanism to control unresolved convective activity and its transfer to the resolved finite element scales. {In the analysis, however, this mechanism is clearly identified for the linearised problem, but estimates for the fully non-linear problem are only meaningful in viscous-dominated flows. Nevertheless,} this interpretation has been explored in orthogonal subscale formulations for large-eddy simulation \cite{GuaschCodina2013} and is consistent with recent discussions on dissipation-aware finite element methods for turbulent flows~\cite{FehnKronbichlerLube2025}. 

Dynamic subscales, introduced in \cite{codina9}, are particularly relevant in transient problems. In fact, for the transient Stokes problem, using standard stabilised methods, or quasi-static subscales, stability could only be proven under the `reverse' CFL condition $\Delta t\ge Ch^2$, where $\Delta t$ is the time step, and $h$ the mesh size (see, e.g., \cite{BOCHEV20041471,BARRENECHEA2007219,codinap1} for discussions of this topic). This restriction leads to a loss of robustness under anisotropic time-space discretisations.
This was one motivating factor to retain the subscale time derivative, as it restores stability and yields estimates that are robust with respect to such anisotropy \cite{BadiaCodina2009}. Building on this insight, a dynamic, term-by-term VMS formulation was later developed for the incompressible Navier--Stokes equations, using orthogonal subgrid scales and separating the convective and pressure-gradient contributions at the subscale level \cite{Codina2008, cowenc2, CastilloCodina2019}. This non-residual viewpoint has also been used in semi-discrete analyses of other time-dependent fluid models \cite{BarrenecheaCastillo2019}. More recently, dynamic non-residual VMS formulations have been assessed computationally for generalised Newtonian flows with high-order spatial and temporal approximations \cite{GuerreroCastillo2025}, and term-by-term dynamic subscales have been used in large-scale incompressible turbulent aerodynamics to obtain robust under-resolved simulations without introducing a problem-specific turbulence model \cite{Escobar2026}.

These developments motivate the present analysis, but they also make  the remaining gap clear. What is missing is not a stabilised approximation of generalised Newtonian flows, nor a dynamic VMS formulation for Newtonian Navier--Stokes flows. The missing point is a semi-discrete stability and error analysis showing how a dynamic non-residual separation of the unresolved momentum mechanisms controls the pressure-gradient contribution in the non-linear shear-dependent case. 

The method analysed in this paper is a non-residual variational multiscale formulation with orthogonal projections and dynamic subscales. The unresolved contributions are introduced term by term, rather than through a single momentum residual. This separation makes it possible to distinguish the pressure-divergence, convective and pressure-gradient mechanisms in the stability analysis. The formulation contains two dynamic velocity subscales, associated respectively with the convective contribution and the pressure-gradient contribution, together with a pressure stabilisation acting on the discrete divergence. The stabilisation parameters depend on the mesh size, a characteristic viscosity and the resolved convective field.

The main contribution of this paper is an unconditional semi-discrete 
stability analysis for this non-residual VMS formulation applied to a non-linear class of generalised Newtonian Navier--Stokes flows. {By `unconditional' in the semi-discrete setting we mean that there is no limitation on the time interval of analysis, contrary to what was found in \cite{BarrenecheaCastillo2019} using quasi-static subscales.} For a linearised problem, we first prove well-posedness and derive an energy estimate for the resolved velocity and the dynamic subscales. We then establish stability in an anisotropic VMS norm that controls the viscous dissipation, the subscale energies, the discrete divergence and a coupled acceleration-convection-pressure balance. The use of dynamic subscales allows us to analyse the semi-discrete method (i.e., before time discretisation), which would not be possible had quasi-static subscales been used.
For the non-linear problem, we use a fixed-point argument and obtain optimal-order a priori bounds under suitable regularity and smallness assumptions. A key point in the proof is the role of the pressure-related dynamic subscale: its time derivative provides the additional control needed to absorb the projected pressure-gradient error, a mechanism that is absent in quasi-static subscale formulations and would otherwise require an additional large-time assumption, as was required in \cite{BarrenecheaCastillo2019}. We also prove an additional weak-in-time estimate for the discrete pressure-convection balance, {but it also provides spatial control on both the pressure gradient and the convective term. This estimate is only meaningful for mildly non-linear viscosities, but in particular applies to Newtonian flows, a setting in which it is also a new result.}

The paper is organised as follows. Section~\ref{sec:notation_preliminaries} introduces the notation, functional setting, finite element spaces and projection tools used in the analysis. Section~\ref{sec:problem} presents the generalised Newtonian model, its weak formulation and the Galerkin finite element approximation. Section~\ref{sec:vms} defines the non-residual VMS formulation with dynamic subscales and the corresponding stabilisation parameters. Section~\ref{sec:analysis} proves well-posedness and unconditional semi-discrete stability for the linearised problem, and develops the fixed-point argument for the non-linear formulation. Section~\ref{sec:pressure_convection_estimate} establishes an additional weak-in-time bound for the discrete pressure-gradient and convective contributions. Section~\ref{sec:conclusions} gives the conclusions.

\section{Notation and preliminaries}\label{sec:notation_preliminaries}

\subsection{Notation}\label{subsec:notation}

We introduce the notation used throughout the paper. 
We adopt a standard notation for Lebesgue and Sobolev spaces. In particular, for $\omega\subset\mathbb{R}^d $ with $d=2,3$, $m\in\mathbb{N}$ and $p\ge 1$, $W^{m,p}(\omega)$ denotes the Sobolev space of functions whose distributional derivatives up to order $m$ belong to $L^p(\omega)$, endowed with its standard norm, hereafter denoted $\double{\cdot}_{m,p,\omega}$. 
For $p=2$, we write $W^{m,2}(\omega)=H^m(\omega)$. The space $H^1_0(\omega)$ consists of functions in $H^1(\omega)$ with vanishing trace on $\partial\omega$. The topological dual of $H^1_0(\omega)$ is denoted by $H^{-1}(\omega)$, and the corresponding duality pairing by $\langle\cdot,\cdot\rangle_\omega$. The $L^2(\omega)$ inner product, for scalar, vector or tensor fields, is denoted by $(\cdot,\cdot)_\omega$. 

The following notation is used for norms over the computational domain $\Omega$: $\double{\cdot}$ denotes the $L^2(\Omega)$-norm, $\double{\cdot}_{L^4}$ the $L^4(\Omega)$-norm, $\double{\cdot}_{\infty}$ the $L^\infty(\Omega)$-norm, $\double{\cdot}_{k}$ the $H^k(\Omega)$-norm, and $\double{\cdot}_{\ell,p}$ the $W^{\ell,p}(\Omega)$-norm. The corresponding semi-norm in $W^{\ell,p}(\Omega)$ is denoted by $|\cdot|_{\ell,p}$. The inner product in $L^2(\Omega)$ will be denoted by $(\cdot,\cdot)$.

The symbol $\lesssim$ denotes an inequality up to a positive constant independent of the mesh size and of the physical parameters. Any relevant dependence on physical constants will be stated explicitly. For a vector field $\bv$, its symmetric gradient is denoted by $\grad^s\bv$ and defined as
\[
\grad^s\bv=\frac{1}{2}\left(\grad\bv+(\grad\bv)^T\right).
\]
The time derivative is denoted by $\partial_t$.

\subsection{Preliminaries}\label{subsec:preliminaries}

We first recall some standard inequalities. For every $\bv\in H^1_0(\Omega)^d$, the full gradient can be bounded in terms of the symmetric gradient as
\begin{equation}
	\double{\grad\bv}\le\sqrt{2}\double{\grad^s\bv}.
	\label{eq:korn_simple}
\end{equation}
Moreover, the Poincaré-Friedrichs and Korn inequalities imply that there exist positive constants $c_{\Omega}$ and $c_K$, depending only on $\Omega$, such that
\begin{equation}
	\double{\bv}\le c_{\Omega}\double{\grad\bv},
	\qquad
	\double{\bv}_{1}\le c_K\double{\grad^s\bv}
	\qquad \forall \bv\in H^1_0(\Omega)^d .
	\label{eq:poincare_korn}
\end{equation}
For a proof of the Poincar\'e inequality see, e.g., \cite[Theorem~8.3-3]{Ciarlet-2025}, and for a proof of
Korn's inequality, see \cite[Theorem~8.17-4]{Ciarlet-2025}.
We shall also use the Sobolev embedding (see, e.g., \cite[Theorem~8.4-1]{Ciarlet-2025})
\begin{equation}
	\double{\bv}_{L^4}\le C_{\rm emb} \double{\bv}_{1}
	\qquad \forall \bv\in H^1_0(\Omega)^d ,
	\label{eq:embedding_L4_H1}
\end{equation}
where the constant $C_{\rm emb}$ depends only on $\Omega$; see, for instance, \cite[Theorem~8.4-1]{Ciarlet-2025}. 

Since the discrete velocity field is not pointwise divergence-free in general, we use the skew-symmetric form of the convective term. For $\bu,\bv,\bw\in H^1_0(\Omega)^d$, we define
\[
\overline{c}(\bu,\bv,\bw)=\int_\Omega\mathcal{N}(\bu,\bv)\cdot\bw\ddx,
\qquad
\mathcal{N}(\bu,\bv)=\bu\cdot \grad \bv+\frac{1}{2}(\div\bu)\bv .
\]
This form satisfies the following properties (see \cite[Chapter~3]{BGHRR-2024}):
\begin{enumerate}
	\item $\overline{c}(\bu,\bv,\bv)=0$ for every $\bu,\bv\in H^1_0(\Omega)^d$.
	\item There holds
	\begin{equation}
		|\overline{c}(\bu,\bv,\bw)|
		\le \left(C_{\rm emb}\right)^2
		\double{\bu}_{1}\double{\bv}_{1}\double{\bw}_{1}
		\qquad
		\forall \bu,\bv,\bw\in H^1_0(\Omega)^d .
		\label{eq:convective_bound_H1}
	\end{equation}
	\item If $\bu\in W^{1,\infty}(\Omega)^d$, $\bv\in H^1_0(\Omega)^d$ and $\bw\in L^2(\Omega)^d$, then
	\begin{equation}
		|\overline{c}(\bu,\bv,\bw)|
		\le
		\double{\bw}
		\left(
		\double{\bu}_{\infty}\double{\grad\bv}
		+
		\frac{1}{2}\double{\div\bu}_{\infty}\double{\bv}
		\right).
		\label{eq:convective_bound_Linfty}
	\end{equation}
\end{enumerate}

Let $\mathcal{T}_h=\{K\}$ be a finite element partition of $\Omega$. The diameter of an element $K\in\mathcal{T}_h$ is denoted by $h_K$, and the mesh size is $h=\max\{h_K:K\in\mathcal{T}_h\}$. We consider quasi-uniform families of meshes, so that all element diameters are bounded above and below by constants proportional to $h$.

Let, for $k\ge 1$, $\Sh$ be the space given by
\begin{equation}\label{FE-space-generic}
	\Sh:=\{\bv_h\in C(\overline{\Omega}):\ \bv_h|_K\in P_k(K)\quad \forall K\in\mathcal{T}_h\},
\end{equation}
where $P_k(K)$ denotes the space of polynomials of degree at most $k$ on $K$. The polynomial order $k$ will be considered fixed, i.e., we only consider convergence under $h$-refinement.  We shall use the following inverse inequality: there exists a constant $c_{\rm inv}$, independent of $h$, such that
\begin{equation}
	\double{\bv_h}_{\ell,p,K}
	\le
	c_{\rm inv}
	h_K^{m-\ell+d\left(\frac{1}{p}-\frac{1}{q}\right)}
	\double{\bv_h}_{m,q,K},
	\label{eq:inverse_inequality}
\end{equation}
for all $\ell\ge m$, $p,q\in[1,+\infty]$, all $K\in\mathcal{T}_h$, and all finite element functions $\bv_h$ defined on $K$; see, for instance, \cite[Lemma~12.1]{EG21-I}. 

We also use standard approximation estimates for the $L^2$-projection. If $\bv\in W^{\ell,p}(\Omega)$ and $P[\bv]\in \Sh$ denotes its $L^2$-projection onto $\Sh$, then there exists a constant $C$, independent of $h$, such that
\begin{equation}
	\double{\bv-P[\bv]}_{p}
	\le
	C h^\ell |\bv|_{\ell,p},
	\qquad
	0\le \ell\le k+1,\quad 1\le p\le \infty .
	\label{eq:interpolation_estimate}
\end{equation}
Further estimates of this type can be found in \cite[Chapter~11]{EG21-I}. In particular, using the $W^{1,\infty}(\Omega)$-stability of the $L^2$-projection on finite element spaces (\cite{Crouzeix1987} for $d=2$, and \cite{Diening2021} for $d=3$), we obtain the bound
\begin{equation}
	\double{\grad(\bv-P[\bv])}_{\infty}
	\le
	C\, \double{\bv}_{1,\infty},
	\label{eq:projection_W1infty}
\end{equation}
where $C$ is independent of $h$.

We shall use $P_u$ to denote the $L^2(\Omega)$-orthogonal projection onto the finite element space $\Sh^d$, where $\Sh$ is defined in \eqref{FE-space-generic} (which is related to the velocity discretisation),   
and $P_u^\perp := I-P_u$. Analogously, $P_p$ denotes the $L^2(\Omega)$-orthogonal projection onto $\Sh$ (which is the space that will be used to discretise the pressure), and
$P_p^\perp := I-P_p$.

\section{Continuous problem and finite element formulation}\label{sec:problem}

\subsection{Governing equations}\label{subsec:governing_equations}

Let $\Omega\subset\mathbb{R}^d$, $d=2,3$, be a bounded polyhedral domain occupied by the fluid, with boundary $\partial\Omega$. Let $(0,T)$ be the time interval of interest, with $0<T\le\infty$. We consider an incompressible, isothermal generalised Newtonian flow governed by
\begin{align}
	\partial_t\bu+\bu\cdot\grad\bu-\div\left(2\mu\left(\cdot,|\grad^s\bu|\right)\grad^s\bu\right)+\grad p &= \bff
	&&\text{in }\Omega\times(0,T], \label{eq:strong_momentum} \\
	\div\bu &= 0
	&&\text{in }\Omega\times(0,T], \label{eq:strong_continuity}\\
	\bu &= 0
	&&\text{on }\partial\Omega\times(0,T], \label{eq:strong_boundary}\\
	\bu(0) &= \bu_0
	&&\text{in }\Omega. \label{eq:strong_initial}
\end{align}
Here $\bu$ is the velocity field, $p$ is the pressure, $\bff$ is a prescribed body force, and $|\grad^s\bu|$ denotes the Frobenius norm of the symmetric gradient. The function $\mu(\bx,|\grad^s\bu|)$ is the apparent viscosity. {Note that we have taken the density in the momentum equation as $\rho = 1$, so that $\mu$ has to be understood as a kinematic viscosity rather than the dynamic one.}

We assume that $\mu$ satisfies the following structural conditions.

\begin{assumption}\label{ass:viscosity_lipschitz}
	There exists a constant $C_\mu>0$ such that
	\begin{equation}
		|\mu(\bx,t)-\mu(\bx,s)|\le C_\mu |t-s|
		\qquad
		\forall t,s\in\mathbb{R}_{\ge0},\quad \bx\in\overline{\Omega}.
		\label{eq:mu_lipschitz}
	\end{equation}
\end{assumption}
\begin{assumption}\label{ass:viscosity_bound} There exist constants $\mu_0>\mu_\infty>0$ such that
	\begin{equation}
		\mu_\infty\le \mu(\bx,t)\le \mu_0
		\qquad
		\forall t\in\mathbb{R}_{\ge0},\quad \bx\in\overline{\Omega}.
		\label{eq:mu_bounds}
	\end{equation}
\end{assumption}
Several regularised rheological laws satisfy these assumptions. Representative examples are listed in Table~\ref{tab:viscosity_models}. Note that in the models considered, the viscosity is bounded even if the velocity gradient is not. In the Eyring model, the value at the origin is understood by continuous extension. 

\begin{table}[t]
	\centering
	\begin{tabular}{ll}
		\hline
		Model  & $\mu(\bx,|\grad^s\bu|)$ \\
		\hline
		Carreau \cite{Carreau1972}
		& $\displaystyle \mu_\infty+\frac{\mu_0-\mu_\infty}{(1+\lambda^2|\grad^s\bu|^2)^{(1-n)/2}}$ \\[0.25cm]
		
		Carreau-Yasuda \cite{CarreauYasuda}
		& $\displaystyle \mu_\infty+\frac{\mu_0-\mu_\infty}{(1+\lambda^b|\grad^s\bu|^b)^{(1-n)/b}}$ \\[0.25cm]
		
		Cross \cite{Cross1965}
		& $\displaystyle \mu_\infty+\frac{\mu_0-\mu_\infty}{1+\lambda^r|\grad^s\bu|^{r}}$ \\[0.25cm]
		
		Quemada \cite{Quemada1977}
		& $\displaystyle \mu_\infty\left(\frac{1+\lambda^r|\grad^s\bu|^{r}}{(\mu_0/\mu_\infty)^{1/2}+\lambda^r|\grad^s\bu|^{r}}\right)^2$ \\[0.25cm]
		
		Eyring \cite{Eyring1936}
		& $\displaystyle \mu_\infty+(\mu_0-\mu_\infty)\frac{\operatorname{arcsinh}(\lambda|\grad^s\bu|)}{\lambda|\grad^s\bu|}$ \\
		\hline
	\end{tabular}
	\caption{Regularised apparent-viscosity laws covered by assumptions \textbf{(A1)}-\textbf{(A2)}. Here $\mu_\infty,\mu_0,r,\lambda$ and $b$ are positive constants and $0\le n\le 1$. 
	}
	\label{tab:viscosity_models}
\end{table}

It will be useful to write the viscosity as
\begin{equation}
	\mu(\bx,t)=\mu_\infty+\phi(\bx,t),
	\qquad
	t\ge0,\quad \bx\in\overline{\Omega}.
	\label{eq:mu_phi_decomposition}
\end{equation}
Then $\phi\in C(\overline{\Omega}\times\mathbb{R}_{\ge0};\mathbb{R}_{\ge0})$ is Lipschitz continuous with respect to its second variable, with Lipschitz constant $C_\mu$, and
\begin{equation}
	0\le\phi(\bx,t)\le \mu_0-\mu_\infty,
	\qquad
	t\ge0,\quad \bx\in\overline{\Omega}.
	\label{eq:phi_bounds}
\end{equation}

\subsection{Weak formulation}\label{subsec:weak_formulation}

Let $\LV=H_0^1(\Omega)^d$, $Q=L^2(\Omega)/\mathbb{R}$, and $\LX=\LV\times Q$. The weak formulation of \eqref{eq:strong_momentum}-\eqref{eq:strong_initial} consists of finding $\BU=(\bu,p):(0,T)\to\LX$, satisfying the initial condition, such that
\begin{align}
	(\partial_t\bu,\bv)
	+\overline{c}(\bu,\bu,\bv)
	+(2\mu(|\grad^s\bu|)\grad^s\bu,\grad^s\bv)
	-(p,\div\bv)
	&= \langle\bff,\bv\rangle,
	\label{eq:weak_momentum}\\
	(q,\div\bu)
	&=0,
	\label{eq:weak_continuity}
\end{align}
for all $\BV=(\bv,q)\in\LX$ a.e. in $(0,T)$. The forcing term is assumed to belong to a suitable dual space, for instance $L^2(0,T;\LV')$.

For notational convenience, we introduce the form
\begin{align}
	B(\widehat{\bu};\BU,\BV)
	&=
	\overline{c}(\widehat{\bu},\bu,\bv)
	+
	(2\mu(|\grad^s\widehat{\bu}|)\grad^s\bu,\grad^s\bv)-(p,\div\bv)
	+
	(\div\bu,q).
	\label{eq:B_form}
\end{align}
The weak problem can then be written as: Find $U=(\bu,p): (0,T) \rightarrow\LX$ such that
\begin{equation}
	(\partial_t\bu,\bv)+B(\bu;\BU,\BV)=\langle\bff,\bv\rangle,
	\label{eq:compact_weak_problem}
\end{equation}
for all $\BV=(\bv,q)\in \LX$, a.e. in $(0,T)$.
\subsection{Galerkin finite element approximation}\label{subsec:galerkin}

Given the mesh $\mathcal{T}_h$, we construct conforming finite element spaces for the velocity and pressure, $\LV_h\subset\LV$ and $Q_h\subset Q$ respectively, and we define $\LX_h=\LV_h\times Q_h$. The Galerkin approximation of \eqref{eq:compact_weak_problem} consists of finding $\BU_h=(\bu_h,p_h):(0,T)\to\LX_h$ such that
\begin{equation}
	(\partial_t\bu_h,\bv_h)+B(\bu_h;\BU_h,\BV_h)=\langle\bff,\bv_h\rangle
	,
	\label{eq:galerkin_problem}
\end{equation}
for all $\BV_h=(\bv_h,q_h)\in\LX_h$, a.e. in $(0,T)$.

Testing \eqref{eq:galerkin_problem} with $\BV_h=\BU_h$ and using the skew-symmetry of the convective form gives
\begin{equation}
	B(\bu_h;\BU_h,\BU_h)
	=
	\int_\Omega 2\mu(|\grad^s\bu_h|)|\grad^s\bu_h|^2\,d\bx
	\ge
	2\mu_\infty\double{\grad^s\bu_h}^2 .
	\label{eq:galerkin_energy}
\end{equation}
Thus, the Galerkin form controls the viscous dissipation, but does not provide stability in the pressure unless the  following discrete inf-sup condition holds:
\begin{equation}
	\inf_{q_h\in Q_h\backslash\{0\}}
	\sup_{\bv_h\in\LV_h\backslash\{0\}}
	\frac{(q_h,\div\bv_h)}
	{\double{q_h}\double{\bv_h}_{1}}
	\ge \beta,
	\label{eq:discrete_infsup}
\end{equation}
with $\beta>0$ independent of $h$. In convection-dominated regimes, additional stabilisation is also needed to control unresolved velocity scales. The method analysed below addresses these issues through a non-residual VMS construction with dynamic orthogonal subscales.

\subsection{Non-residual VMS formulation with dynamic subscales}\label{sec:vms}

The stabilised method separates the unresolved contributions associated with the pressure-divergence, convective and pressure-gradient terms. The formulation of the method, proposed by Castillo \& Codina \cite{CastilloCodina2019} for the Navier--Stokes equations, is described next.

We consider continuous and equal-order interpolation spaces for the velocity and pressure, $\LV_h$ and $Q_h$ respectively. Thus, the approximation spaces considered are given by $\LV_h=\left[\Sh\cap H_0^1(\Omega)\right]^d$, $Q_h= \Sh/\mathbb{R}$ and we define $\LX_h=\LV_h\times Q_h$. {Nevertheless, any other conforming approximation could be accommodated; in the case of using discontinuous pressures, interelement boundary terms would be required. We restrict ourselves to equal-order continuous interpolations for the sake of conciseness.}

We introduce a pressure subscale $\tilde p$ and two velocity subscales, $\tilde\bu_1$ and $\tilde\bu_2$, defined by
\begin{align}
	\tau_2^{-1}(\bu_h)\tilde p
	&=
	-P_p^\perp[\div\bu_h],
	\label{eq:pressure_subscale_pointwise}\\
	\partial_t\tilde\bu_1+\tau_1^{-1}(\bu_h)\tilde\bu_1
	&=
	-P_u^\perp[\mathcal{N}(\bu_h,\bu_h)],
	\label{eq:convective_subscale_pointwise}\\
	\partial_t\tilde\bu_2+\tau_1^{-1}(\bu_h)\tilde\bu_2
	&=
	-P_u^\perp[\grad p_h].
	\label{eq:pressure_gradient_subscale_pointwise}
\end{align}
Here $P_u^\perp$ and $P_p^\perp$ are the unresolved velocity and pressure projections introduced in Section~\ref{subsec:preliminaries}. The stabilisation parameters are
\begin{equation}
	\tau_1(\widehat{\bu}_h)
	=
	\left(
	c_1\frac{\mu_c}{h^2}
	+
	c_2\frac{\double{\widehat{\bu}_h}_{\infty}}{h}
	\right)^{-1},
	\qquad
	\tau_2(\widehat{\bu}_h)
	=
	\frac{h^2}{c_1\tau_1(\widehat{\bu}_h)},
	\label{eq:stabilisation_parameters}
\end{equation}
where $c_1,c_2>0$ are algorithmic constants, and $\mu_c$ is a characteristic viscosity. {Note that $\tau_1$ and $\tau_2$ are time-dependent functions, i.e., \eqref{eq:stabilisation_parameters} holds a.e. in $(0,T)$. Likewise, let us remark that when considering refinement with respect to the interpolation order, $c_1$ needs to scale as $k^4$ and $c_2$ as $k$~\cite{Codina2018}.}

\begin{remark}\label{rem:characteristic_viscosity}
	The characteristic viscosity is assumed to satisfy
	$\mu_\infty\le\mu_c\le\mu_0$. A natural choice is the spatial average of the apparent viscosity,
	\[
	\mu_c(\widehat{\bu}_h)
	=
	\frac{1}{|\Omega|}
	\int_\Omega
	\mu(\bx,|\grad^s\widehat{\bu}_h|)\,d\bx .
	\]
	
\end{remark}

	\begin{remark}
		When passing to a fully discrete scheme, the non-linear dependence of $\tau_1$, $\tau_2$
		and $\mu_c$ on $\widehat{u}_h$ requires a consistent temporal extrapolation. First-order
		extrapolations of the form
		\[
		\mu_{n+1}\nabla^T\! u_{n+1} \approx \sqrt{\mu_{n+1}}\sqrt{\mu_n}\,\nabla^T\! u_n,
		\]
		shown in \cite{BarrenecheaCastilloPacheco2024} to preserve unconditional stability for
		variable-viscosity Navier--Stokes problems, could be adapted to the stabilisation
		parameters \eqref{eq:stabilisation_parameters} without requiring \emph{a priori} knowledge of
		$\mu(x,|\nabla^s u_h|)$ at the new time level.
	\end{remark}
	
	Let $\{\phi_i\}_{i=1}^{n_u}$ and $\{\varphi_k\}_{k=1}^{n_p}$ denote the bases used herein for the discrete velocity and pressure spaces. We define
	\[
	\widetilde V_1
	=
	\operatorname{span}\left\{
	P_u^\perp[\mathcal{N}(\phi_i,\phi_j)]
	\right\}_{i,j=1}^{n_u},
	\qquad
	\widetilde V_2
	=
	\operatorname{span}\left\{
	P_u^\perp[\grad\varphi_k]
	\right\}_{k=1}^{n_p}.
	\]
	Both spaces are finite-dimensional subspaces of $L^2(\Omega)^d$. We denote by $\widetilde P_1$ and $\widetilde P_2$ the $L^2(\Omega)$-orthogonal projections onto $\widetilde V_1$ and $\widetilde V_2$, respectively.
	
	The finite element method analysed in this work reads as follows: Find 
	$(\bu_h,p_h,\tilde\bu_1,\tilde\bu_2):(0,T)
	\to
	\LV_h\times Q_h\times\widetilde V_1\times\widetilde V_2$
	such that
	\begin{align}
		(\partial_t\bu_h,\bv_h)
		+(2\mu(\bx,|\grad^s\bu_h|)\grad^s\bu_h,\grad^s\bv_h)
		+\overline{c}(\bu_h,\bu_h,\bv_h)
		-(p_h,\div\bv_h)
		\nonumber\\
		\quad
		+(\tau_2\div\bv_h,P_p^\perp[\div\bu_h])
		-\overline{c}(\bu_h,\bv_h,\tilde\bu_1)
		&=
		\langle\bff,\bv_h\rangle,
		\label{eq:vms_momentum}\\
		(q_h,\div\bu_h)
		-(\grad q_h,\tilde\bu_2)
		&=
		0,
		\label{eq:vms_continuity}\\
		(\partial_t\tilde\bu_1,\tilde\bv_1)
		+\tau_1^{-1}(\tilde\bu_1,\tilde\bv_1)
		+\overline{c}(\bu_h,\bu_h,\tilde\bv_1)
		&=
		0,
		\label{eq:vms_convective_subscale}\\
		(\partial_t\tilde\bu_2,\tilde\bv_2)
		+\tau_1^{-1}(\tilde\bu_2,\tilde\bv_2)
		+(\grad p_h,\tilde\bv_2)
		&=
		0,
		\label{eq:vms_pressure_gradient_subscale}
	\end{align}
	for all
	$(\bv_h,q_h,\tilde\bv_1,\tilde\bv_2)\in
	\LV_h\times Q_h\times\widetilde V_1\times\widetilde V_2$ a.e. in $(0,T)$. We note that in the above formulation the pressure subscale has already been `condensed', leading to the penultimate term on the LHS of \eqref{eq:vms_momentum}.

	This system is supplemented with
	\begin{equation}
		\bu_h(0)=P_{u,0}[\bu_0],
		\qquad
		\tilde\bu_1(0)=\tilde P_1[\bu_0],
		\qquad
		\tilde\bu_2(0)=\tilde P_2[\bu_0],
		\label{eq:vms_initial_conditions}
	\end{equation}
	where $P_{u,0}$ denotes the $L^2(\Omega)$-projection onto $\LV_h$.

	The pointwise subscale equations \eqref{eq:convective_subscale_pointwise}-\eqref{eq:pressure_gradient_subscale_pointwise} are equivalent to the variational subscale equations \eqref{eq:vms_convective_subscale}-\eqref{eq:vms_pressure_gradient_subscale} after projection onto $\widetilde V_1$ and $\widetilde V_2$. Since $\tau_1$ is constant in space for a fixed $\bu_h$, these equations can be written as
	\[
	\widetilde P_1[\partial_t\tilde\bu_1+\tau_1^{-1}\tilde\bu_1]
	=
	-\widetilde P_1[\mathcal{N}(\bu_h,\bu_h)],
	\qquad
	\widetilde P_2[\partial_t\tilde\bu_2+\tau_1^{-1}\tilde\bu_2]
	=
	-\widetilde P_2[\grad p_h],
	\]
	which recovers \eqref{eq:convective_subscale_pointwise}-\eqref{eq:pressure_gradient_subscale_pointwise} in the corresponding subscale spaces.

	The formulation differs from residual-based VMS methods in that the unresolved contributions are not introduced through a single momentum residual. Instead, the pressure-divergence, convective, and pressure-gradient mechanisms are separated before the subscale equations are defined. This distinction is relevant for the analysis. It allows the pressure-gradient contribution to be associated with a specific dynamic velocity subscale, and the divergence defect with a separate pressure stabilisation. Hence, the stability proof can identify which unresolved component controls each part of the discrete momentum balance.

	The next result, proved in \cite{Codina1997} {for some particular cases}, will be very useful in the analysis presented below: there exists a constant $\beta_0>0$, independent of $h$, such that
	\begin{equation}
		\double{\bz_h}
		\le
		\frac{1}{\beta_0}
		\double{P_{u,0}[\bz_h]+P_u^\perp[\bz_h]}
		\le
		\frac{1}{\beta_0}
		\left(
		\double{P_{u,0}[\bz_h]}
		+
		\double{P_u^\perp[\bz_h]}
		\right),
		\label{eq:interpolation_condition}
	\end{equation}
	for every piecewise polynomial vector field $\bz_h$. Here $P_{u,0}[\bz_h]$ denotes the $L^2(\Omega)$-projection of $\bz_h$ onto the finite element space $\LV_h$. {In general, we will accept this inequality as an assumption. It can be expressed as a mild inf-sup condition that holds for most finite element partitions.}
	
	\section{Numerical analysis}\label{sec:analysis}
	
	The purpose of this section is to identify the stability mechanisms provided by the dynamic non-residual VMS formulation. The analysis is carried out at the semi-discrete level, that is, no time discretisation is performed.
	The linearised problem isolates the role of the stabilisation parameters and the dynamic subscales. In contrast, the non-linear problem is treated through a fixed-point argument in a stabilised norm adapted to the pressure-gradient subscale.
	
	\subsection{Linearised semi-discrete problem}\label{subsec:linearised_problem}
	
	We first consider a linearised semi-discrete problem. The role of this step is twofold. On the one hand, it separates the stability properties of the VMS formulation from the non-linear fixed-point argument. On the other hand, it shows how the dynamic subscales enter the energy balance before any time-discretisation effect is present. Let $\widehat{\bu}_h:[0,T]\to\LV_h$ be given. The linearised counterpart of
	\eqref{eq:vms_momentum}-\eqref{eq:vms_pressure_gradient_subscale} is obtained
	by freezing the convective field and the apparent viscosity at
	$\widehat{\bu}_h$. It consists of finding
	$\BU_h=[\bu_h,p_h,\tilde\bu_1,\tilde\bu_2]:[0,T]\to \LV_h\times Q_h\times \widetilde V_1\times \widetilde V_2$ such that:
	\begin{align}
		(\partial_t\bu_h,\bv_h)
		+(2\mu(|\grad^s\widehat{\bu}_h|)\grad^s\bu_h,\grad^s\bv_h)
		+\overline{c}(\widehat{\bu}_h,\bu_h,\bv_h)
		-(p_h,\div\bv_h)
		\nonumber\\
		\quad
		+(\tau_2\div\bv_h,P_p^\perp[\div\bu_h])
		-\overline{c}(\widehat{\bu}_h,\bv_h,\tilde\bu_1)
		&=
		\langle\bff,\bv_h\rangle,
		\label{eq:linear_momentum}\\
		(q_h,\div\bu_h)-(\grad q_h,\tilde\bu_2)
		&=
		0,
		\label{eq:linear_continuity}\\
		(\partial_t\tilde\bu_1,\tilde\bv_1)
		+\tau_1^{-1}(\tilde\bu_1,\tilde\bv_1)
		+\overline{c}(\widehat{\bu}_h,\bu_h,\tilde\bv_1)
		&=
		0,
		\label{eq:linear_convective_subscale}\\
		(\partial_t\tilde\bu_2,\tilde\bv_2)
		+\tau_1^{-1}(\tilde\bu_2,\tilde\bv_2)
		+(\grad p_h,\tilde\bv_2)
		&=
		0,
		\label{eq:linear_pressure_gradient_subscale}
	\end{align}
	for all
	$\BV_h=(\bv_h,q_h,\tilde\bv_1,\tilde\bv_2)\in\LV_h\times Q_h\times \widetilde V_1\times \widetilde V_2$ a.e. in $(0,T)$ and supplemented with the following initial condition
	\begin{equation}
		\bu_h(0)=P_{u,0}[\bu_0],
		\qquad
		\tilde\bu_1(0)=\tilde P_1[\bu_0],
		\qquad
		\tilde\bu_2(0)=\tilde P_2[\bu_0].
		\label{eq:linear_initial_conditions}
	\end{equation}
	
	Throughout this section, unless otherwise stated, we write
	$\tau_1=\tau_1(\widehat{\bu}_h)$ and $\tau_2=\tau_2(\widehat{\bu}_h)$.
	
	\subsection{Existence and uniqueness of the linearised semi-discrete problem}\label{subsec:linear_wellposedness}
	
	Let
	$V^*:=\LV_h\oplus\widetilde V_2$. We define $J^*\subset V^*$ as the finite-dimensional subspace such that
	\begin{equation}
		J^*:=\left\{\bv^*=\bv_h+\tilde\bv_2\in V^*:\bv_h\in\LV_h,\,\,\tilde\bv_2\in \widetilde V_2\text{ and }(q_h,\div\bv_h)-(\grad q_h,\tilde\bv_2)=0\quad\forall q_h\in Q_h \right\}.
	\end{equation}
	We also set
	$\bu^*:=\bu_h+\tilde\bu_2$.
	
	The introduction of $J^*$ is a technical device that allows the pressure to be eliminated from the linearised system. The variable $\bu^*$ contains the resolved velocity and the pressure-gradient subscale, which are coupled through the discrete incompressibility equation. This reformulation reduces the problem to an ordinary differential equation on a finite-dimensional constrained space, after which the pressure can be recovered separately.
	
	\begin{remark}\label{rem:Jstar_nonempty}
		The space $J^*$ is non-empty under the compatibility property of the
		orthogonal subscale construction; see, for instance,
		\cite[Theorem 3]{Codina2008}.
	\end{remark}
	
	\begin{theorem}[Existence and Uniqueness]
		Let $\bff\in L^2(0,T;H^{-1}(\Omega)^d)$ and $\bu_0\in L^2(\Omega)^d$. Then, the semi-discrete problem \eqref{eq:linear_momentum}-\eqref{eq:linear_pressure_gradient_subscale} admits a unique solution satisfying $\bu_h,\tilde\bu_1,\tilde\bu_2\in L^\infty(0,T;L^2(\Omega)^d)$ and $\bu_h\in L^2(0,T;H^1(\Omega)^d)$, and the following estimate holds:
		\begin{align}
			\sup_{0\le t\le T}\double{\bu_h(t)}^2+\sup_{0\le t\le T}\double{\tilde\bu_1(t)}^2+\sup_{0\le t\le T}\double{\tilde\bu_2(t)}^2+\mu_{\infty}\int_0^T\double{\grad^s\bu_h}^2\dds & \nonumber\\
			\lesssim \frac{c_K^2}{\mu_{\infty}}\int_0^T\double{\bff}^2_{H^{-1}}\dds+\double{\bu_{0}}^2.&
		\end{align}
	\end{theorem}
	
	\begin{proof} 
		We follow the standard approach of first reformulating the stabilised semi-discrete problem \eqref{eq:linear_momentum}-\eqref{eq:linear_pressure_gradient_subscale} as an equation only for the velocity variables. To this end, we note that $\LV_h\cap\widetilde V_2=\{\boldsymbol{0}\}$, and thus any function $\bv^*\in V^*$ admits the unique decomposition $\bv^*=P_{u,0}[\bv^*]+\tilde P_2[\bv^*]$. 
		
		System \eqref{eq:linear_momentum}-\eqref{eq:linear_pressure_gradient_subscale} can be stated as follows: find $\bu^*\in H^1(0,T;J^*)$ and $\tilde\bu_1\in H^1(0,T;\widetilde V_1)$ such that 
		\begin{align}
			(\partial_tP_{u,0}[\bu^*],P_{u,0}[\bv^*])+(2\mu(|\grad^s\hat\bu_h|)\grad^sP_{u,0}[\bu^*],\grad^sP_{u,0}[\bv^*])\hspace{0.55cm}& \nonumber\\
			+\overline{c}(\hat\bu_h; P_{u,0}[\bu^*],P_{u,0}[\bv^*])+\tau_2(\div P_{u,0}[\bv^*],P_p^\perp[\div P_{u,0}[\bu^*]])&\label{CarreauReducido1}\\
			-\overline{c}(\hat\bu_h,P_{u,0}[\bv^*],\tilde\bu_1)+(\partial_t\tilde P_2[\bu^*],\tilde P_2[\bv^*])+\tau_1^{-1}(\tilde P_2[\bu^*],\tilde P_2[\bv^*])&= \left<\bff,P_{u,0}[\bv^*]\right>, \nonumber\\
			(\partial_t\tilde\bu_1,\bv_1)+(\tau_1^{-1}\tilde\bu_1,\tilde\bv_1)+\overline{c}(\hat\bu_h,P_{u,0}[\bu^*],\tilde \bv_1) &= 0, \label{CarreauReducido2}
		\end{align}
		for any $\bv^*\in J^*$ and $\tilde\bv_1\in\widetilde V_1$ a.e. in $(0,T)$, with the initial condition
		\begin{eqnarray}
			\bu^*(0)=P_{u,0}[\bu_0]+\tilde P_2[\bu_0] \hspace{0.5cm}\text{and}\hspace{0.5cm} \tilde\bu_1(0)=\tilde P_1[\bu_0], \nonumber
		\end{eqnarray}
		which is a standard (linear) Cauchy problem. Then using the theory of ordinary differential equations, we can then prove that $\bu^*(t)$ and $\tilde\bu_1(t)$ do not blow up in some time interval $[0,t_h)$ for $t_h>0$ small enough (see, e.g., \cite[Lemma 1.1, Chapter 3]{Teman1984}).
		Now, we are in position to test \eqref{CarreauReducido1} and \eqref{CarreauReducido2} with $\bv^*=\bu^*$ and $\tilde\bv_1=\tilde\bu_1$, respectively, and using the fact that
		$$\overline{c}(\hat\bu_{h},P_{u,0}[\bu^*],P_{u,0}[\bu^*])=0,$$
		we obtain
		\begin{align}
			\frac{1}{2}\frac{\textrm{d}}{\textrm{d} t}\double{\bu^*}^2+\mu_\infty\double{\grad^s P_{u,0}[\bu^*]}^2+\tau_2\double{P_p^\perp[\div P_{u,0}[\bu^*]]}^2+\tau_1^{-1}&\double{\tilde P_2[\bu^*]}^2\nonumber\\
			+\frac{1}{2}\frac{\textrm{d}}{\textrm{d} t}\double{\tilde\bu_1}^2+\tau_1^{-1}\double{\tilde\bu_1}^2&\le\langle\bff,P_{u,0}[\bu^*]\rangle, 
		\end{align}
		and integrating from $0$ to $t$, and using Korn's and Young's inequalities yields
		\begin{equation}
			\double{\bu^*(t)}^2+\double{\tilde\bu_1(t)}^2+\mu_\infty\int_0^t\double{\grad P_{u,0}[\bu^*]}^2\textrm{d}s\lesssim \frac{c_K^2}{\mu_\infty}\int_0^t\double\bff^2_{H^{-1}}\textrm{d}s+ \double{P_{u,0}[\bu_0]}^2+\double{\tilde P_1[\bu_0]}^2.
		\end{equation}
		
		This energy bound allows us to extend the existence and regularity results for $\bu^*$ over the whole time interval $[0,T]$.
		
		Given $\bu^*(t)$ and $\tilde\bu_1(t)$, the problem for the pressure now reads as follows: find $p_h(t)\in Q_h$ such that
		\begin{align}
			&(p_h,\div P_{u,0}[\bv^*])-(\tilde P_2[\bv^*],\grad p_h) \nonumber \\
			&= (\partial_t\bu^*,\bv^*)+(2\mu(|\grad^s\hat\bu_h|)\grad^sP_{u,0}[\bu^*],\grad^sP_{u,0}[\bv^*])+\overline{c}(\hat\bu_h, P_{u,0}[\bu^*],\bv^*)
			\label{eq:presionequation}\\
			&\quad\ +(\tau_2\div P_{u,0}[\bv^*],P_p^\perp[\div P_{u,0}[\bu^*]])-\overline{c}( \hat\bu_h,P_{u,0}[\bv^*],\tilde\bu_1)\nonumber\\
			&\quad\ -\left<\bff,P_{u,0}[\bv^*]\right>\nonumber \\
			&=:\mathcal{F}(\bv^*) \nonumber
		\end{align}
		a.e. in $(0,T)$, for any $\bv^*\in V^*$. Defining $K^*:=P_{u,0}[\grad Q_h]\oplus P_u^\perp[\grad Q_h]\subset V^*$ we get $J^*=(K^*)^\perp\cap V^*$. In addition, choosing $\bv^*=P_{u,0}(\nabla q_h)+P_h^\perp(\nabla q_h)\in V^*$ in \eqref{eq:presionequation} we see that $p_h$ satisfies the following finite element problem:
		\begin{align}
			(\grad p_h,\grad q_h)=\mathcal{F}(P_{u,0}[\grad q_h]+P_u^\perp[\grad q_h])
		\end{align}
		for all $q_h\in Q_h$. Since the left-hand side of the above problem defines an elliptic bilinear form in $Q_h$, the existence and uniqueness of this problem are a consequence of Lax--Milgram's Lemma. On the other hand, for $\bv^*\in J^*$, both the left-hand side and the
		right-hand side of \eqref{eq:presionequation} vanish. Thus, problem \eqref{CarreauReducido1}-\eqref{CarreauReducido2} has a unique solution $(\bu^*,p_h,\tilde\bu_1)$, and the result follows.
	\end{proof}
	
	\subsection{Stability of the linearised semi-discrete problem}\label{subsec:linear_stability}
	
	The previous estimate provides the basic energy control of the resolved velocity and the dynamic subscales. However, it does not yet display the full stabilising effect of the formulation. In particular, the non-residual VMS terms are designed to control a coupled balance involving acceleration, convection, and pressure gradient, motivating the introduction of the following anisotropic norm
	\begin{align}
		\triple{\BV_h}_{W}^2
		&:=
		\mu_\infty\double{\grad^s\bv_h}^2
		+
		\tau_1^{-1}\double{\tilde\bv_1}^2
		+
		\tau_1^{-1}\double{\tilde\bv_2}^2
		\nonumber\\
		&\quad\ 
		+
		\tau_1
		\double{
			\partial_t(\bv_h+\tilde\bv_1+\tilde\bv_2)
			+
			\mathcal{N}(\widehat{\bu}_h,\bv_h)
			+
			\grad q_h
		}^2
		+
		\tau_2\double{\div\bv_h}^2 .
		\label{eq:working_norm}
	\end{align}
	
	\begin{remark}\label{rem:anisotropic_norm}
		Besides the viscous dissipation and the dissipation of the subscales, this norm controls the coupled balance between acceleration, convection, and pressure gradient. Since the estimate below is obtained before temporal discretisation, no time-step parameter enters the argument. In this sense, the result is the semi-discrete counterpart of an unconditional anisotropic stability bound. {When performing the time discretisation, the time step size and the mesh size $h$ will be allowed to converge to zero at independent rates, i.e., the space-time discretisation will be anisotropic.}
	\end{remark}

	\begin{theorem}[Stability]\label{thm:linear_stability}
		Let $ \BU_h=(\bu_h,p_h,\tilde\bu_1,\tilde\bu_2)$ be the solution to \eqref{eq:linear_momentum}-\eqref{eq:linear_pressure_gradient_subscale}. Suppose that Assumption~\ref{ass:viscosity_bound} and inequality \eqref{eq:interpolation_condition} hold; then we have the following stability bound:
		\begin{align}
			\sup_{0\le t\le T}\double{\bu_h(t)}^2+\sup_{0\le t\le T}\double{\tilde\bu_1(t)}^2+\sup_{0\le t\le T}\double{\tilde\bu_2(t)}^2+\mu_{\infty}\int_0^T\triple{\BU_h}_{W}^2ds& \nonumber\\
			\lesssim \frac{c_K^2}{\mu_{\infty}}\int_0^T\double{\bff}^2_{H^{-1}}ds+\double{\bu_{0}}^2. \label{iq:Stability}
		\end{align}
		Therefore, if $\bff\in L^2(0,T;H^{-1}(\Omega)^d)$ and $\bu_{0}\in L^2(\Omega)^d$, we have that
		$$\bu_h\in L^\infty(0,T;L^2(\Omega)^d),\hspace{0.5cm}\tilde\bu_1\in L^\infty(0,T;L^2(\Omega)^d),\hspace{0.5cm} \tilde\bu_2\in L^\infty(0,T;L^2(\Omega)^d),   $$
		$$\mu_\infty^{
			1/2}\bu_h\in L^2(0,T;H^1(\Omega)^d),\hspace{0.5cm}\tau_1^{-\frac{1}{2}}\tilde\bu_1\in L^2(0,T;L^2(\Omega)^d),\hspace{0.5cm}\tau_1^{-\frac{1}{2}}\tilde\bu_2\in L^2(0,T;L^2(\Omega)^d),$$
		{with bounds independent of the mesh size $h$.}
	\end{theorem}
	
	\begin{proof}
		Taking $(\bv_h,q_h,\tilde\bv_1,\tilde\bv_2)=\BU_{h1}=(\bu_h,p_h,\tilde\bu_1,\tilde\bu_2)$ in (\ref{eq:linear_momentum})-(\ref{eq:linear_pressure_gradient_subscale}) and adding we can write
		\begin{align}
			&(\partial_t\bu_h,\bu_h)+\mu_\infty\double{\grad^s\bu_h}^2+\tau_2\double{P_p^\perp[\div\bu_h]}^2  \nonumber\\
			&+(\partial_t\tilde\bu_1,\tilde\bu_1)+(\partial_t\tilde\bu_2,\tilde\bu_2)+\tau_1^{-1}\double{\tilde\bu_1}^2+\tau_1^{-1}\double{\tilde\bu_2}^2 \le \left<\bff,\bu_h\right>.\label{iq:stab1}
		\end{align}
		Using the inequality $a^2+b^2\ge a^2/3+b^2/3+(a+b)^2/3$ in the last two terms of the left-hand side of (\ref{iq:stab1}), we obtain
		\begin{align}
			(\partial_t\bu_h,\bu_h)+(\partial_t\tilde\bu_1,\tilde\bu_1)+(\partial_t\tilde\bu_2,\tilde\bu_2)+\mu_\infty\double{\grad^s\bu_h}^2+\tau_2\double{P_p^\perp[\div\bu_h]}^2 & \nonumber\\
			+\frac{\tau_1^{-1}}{3}\double{\tilde\bu_1}^2+\frac{\tau_1^{-1}}{3}\double{\tilde\bu_2}^2+\frac{\tau_1^{-1}}{3}\double{\tilde\bu_1+\tilde\bu_2}^2 \le \left<\bff,\bu_h\right>.& \nonumber
		\end{align}

		Let us consider $\BV_{h1}=\tau_1(\bv_1,0,\boldsymbol0,\boldsymbol0)$ with $\bv_1=P_{u,0}\left[\partial_t\bu_h+\mathcal{N}(\hat\bu_h,\bu_h)+\grad p_h\right]$. Testing \eqref{eq:linear_momentum}-\eqref{eq:linear_pressure_gradient_subscale} with $\BV_{h1}$, we have that
		\begin{align*}
			\tau_1(\mu(|\grad^s\hat\bu_h|)\grad^s\bu_h,\grad^s\bv_1)+\tau_1\double{\bv_1}^2&+\tau_1(\tilde\bu_1,P_u^\perp[\mathcal{N}(\hat\bu_h,\bv_1)])\\
			&+\tau_1\tau_2(P_p^\perp[\div\bu_h],\div\bv_1)
			=\tau_1\left<\bff,\bv_1\right>.
		\end{align*}
		Applying Cauchy--Schwarz's and Young's inequality in all products, the inverse inequality \eqref{eq:inverse_inequality}, and using the bound Assumption~\ref{ass:viscosity_bound}, i.e. $|\mu(|\grad^s\hat\bu_h|)|\le \mu_0$, we obtain the following result: 
		\begin{align}
			\tau_1\left<\bff,\bv_1\right>\ge&-\left(\mu_\infty\frac{\double{\grad^s\bu_h}^2}{8}+\frac{4\tau_1^2\double{\bv_1}^2\mu_0^2}{h^2\mu_\infty}\right)-\double{\hat\bu_h}_\infty c_{\rm inv}(1+d)\left(\frac{\double{\tilde\bu_1}^2}{2h}+\tau_1^2\frac{\double{\bv_1}^2}{2h}\right) \nonumber\\
			& -\left(\frac{\tau_2\double{P_p^\perp[\div\bu_h]}^2}{2}+\frac{c_{\rm inv}^2\tau_1^2\tau_2\double{\bv_1}^2}{2h^2}\right)+\tau_1\double{\bv_1}^2.
		\end{align}
		To control $P_p[\div\bu_h]$, we proceed taking $\BV_{h2}=\tau_2(\boldsymbol{0},P_p[\div\bu_h],\boldsymbol{0},\boldsymbol{0})$ in \eqref{eq:linear_momentum}-\eqref{eq:linear_pressure_gradient_subscale}:
		\begin{align*}
			\tau_2(\div\bu_h,P_p[\div\bu_h])\ +&\ \tau_2(\tilde\bu_2,\grad(P_p[\div\bu_h]))\\
			 \ge&\ \tau_2\left(1-c_{\rm inv}^2\tau_1\frac{3\tau_2}{h^2}\right)\double{P_p[\div\bu_h]}^2-\frac{\tau_1^{-1}}{6}\double{\tilde\bu_2}^2. \nonumber
		\end{align*}
		Finally, testing the linear equations \eqref{eq:linear_momentum}-\eqref{eq:linear_pressure_gradient_subscale} with $\BV_h^*:=\BU_{h}+\BV_{h1}+\BV_{h2}$, using the Poincar\'e inequality and \eqref{eq:korn_simple} we get:
		\begin{align*}
			\mu_\infty\frac{\double{\grad^s\bu_h}^2}{2}+\tau_1\left(1-\frac{\tau_1\mu_c}{h^2}\frac{4\mu_0^2}{\mu_\infty\mu_c}-\frac{\tau_1\double{\hat\bu_h}_\infty c_{\rm inv}(1+d)}{2h}-\frac{c_{\rm inv}^2\tau_1\tau_2}{2h^2}\right)\double{\bv_1}^2 &\nonumber\\
			+\frac{\tau_1^{-1}}{3}\double{\tilde\bu_1+\tilde\bu_2}^2+\tau_2\left(1-c_{\rm inv}^2\tau_1\frac{3\tau_2}{h^2}\right)\double{P_p[\div\bu_h]}^2+\frac{\tau_2}{2}\double{P_p^\perp[\div\bu_h]}^2 \nonumber\\
			+\tau_1^{-1}\left(\frac{1}{3}-c_{\rm inv}(1+d)\frac{\tau_1\double{\hat\bu_h}_\infty}{2h}\right)\double{\tilde\bu_1}^2+\frac{\tau_1^{-1}}{6}\double{\tilde\bu_2}^2\\
			\le\frac{c^2_\Omega}{2\mu_\infty}\double{\bff}_{H^{-1}}^2.&
		\end{align*}
		Now using the definition of the stabilisation parameters \eqref{eq:stabilisation_parameters}, we have that
		\begin{equation}
			\frac{\tau_1\mu_c}{h^2}\le c_1^{-1},\hspace{0.5cm}\frac{\tau_1\double{\hat\bu_h}_\infty}{h}\le c_2^{-1},\hspace{0.5cm}\frac{\tau_1\tau_2}{h^2}\le  c_1^{-1}.
		\end{equation}
		Taking the algorithmic constants $c_1,c_2$ large enough, it follows that there exist positive constants $C_j$, $j=1,\ldots,6$, such that
		\begin{align}
			C_1\mu_\infty \double{\grad^s\bu_h}^2+C_2\tau_1\double{P_{u,0}\left[\partial_t(\bu_h+\tilde\bu_1+\tilde\bu_2)+\mathcal{N}(\hat\bu_h,\bu_h)+\grad p_h\right]}^2\nonumber\\
			+C_3\tau_1^{-1}\double{\tilde\bu_1+\tilde\bu_2}^2+C_4\tau_2\double{P_p[\div\bu_h]}^2+\frac{\tau_2}{2}\double{P_p^\perp[\div\bu_h]}^2\nonumber\\
			+C_5\tau_1^{-1}\double{\tilde\bu_1}^2+C_6\tau_1^{-1}\double{\tilde\bu_2}^2 &\le \frac{c_\Omega^2}{2\mu_\infty}\double\bff_{H^{-1}}^2. \nonumber
		\end{align}
		Now to obtain control in $\partial_t(\bu_h+\tilde\bu_1+\tilde\bu_2)+\mathcal{N}(\hat\bu_h,\bu_h)+\grad p_h$, we use the fact that
		$$-\tau_1^{-1/2}(\tilde\bu_1+\tilde\bu_2)=\tau_1^{1/2}P_u^\perp[\partial_t(\bu_h+\tilde\bu_1+\tilde\bu_2)+\mathcal{N}(\hat\bu_h,\bu_h)+\grad p_h].$$
		Thus, the proof follows from the stability \eqref{eq:interpolation_condition}.
	\end{proof}
	
	The estimate in Theorem~\ref{thm:linear_stability} is the key result on stability of this work. Its relevance lies not only in the bound being uniform with respect to the mesh size, but also in its control of the pressure-gradient contribution through the dynamic subscale equation. This point will be used again in the non-linear analysis, where the projected pressure-gradient error is the term that distinguishes the present formulation from a quasi-static subscale approach.
	
	\subsection{Non-linear problem}\label{subsec:nonlinear_problem}
	
	We now consider the non-linear formulation
	\eqref{eq:vms_momentum}-\eqref{eq:vms_pressure_gradient_subscale}. The main difficulty is not the construction of a discrete solution itself, since the discrete problem is finite-dimensional, but rather the derivation of estimates with the same order of convergence as the interpolation error. For this reason, the proof is organised as a fixed-point argument in a stabilised norm that includes the dynamic pressure-gradient contribution, {but not the convective term. Obtaining stability in a norm that provides some control of both pressure and convection in stabilised finite element methods is rare; nevertheless, we will establish a result in this direction in Section~\ref{sec:pressure_convection_estimate}.}
	
	The fixed-point map is built from the linearised problem analysed above, and the central task is to prove that it maps a ball of optimal radius into itself.
	
	We will use the same procedure proposed by Ervin \& Miles \cite{ErvinMiles2003}. The proof can be divided into the following four steps:
	\begin{enumerate}
		\item 
		Define an iterative mapping such that the fixed-point is the solution to the original problem.
		\item Show that the mapping is well-defined and bounded on bounded sets. 
		\item Show that there exists an invariant ball of the mapping.
		\item Apply {Brouwer's  fixed-point} theorem to conclude that a discrete solution exists on this ball.
	\end{enumerate}
	
	\begin{assumption}\label{ass:regular_solution}
		The system \eqref{eq:weak_momentum}-\eqref{eq:weak_continuity} admits a solution $(\bu,p)$, continuous in time and satisfying
		\begin{align}
			\sup_{0\le t\le T}\double{\bu}_{\infty}
			&\le D_1,
			&
			\sup_{0\le t\le T}\double{\grad\bu}_{\infty}
			&\le D_2,
			&
			\sup_{0\le t\le T}\double{\bu}_{k+1}
			&\le D_3,
			\nonumber\\
			\sup_{0\le t\le T}\double{p}_{k}
			&\le D_4,
			&
			\sup_{0\le t\le T}\double{\partial_t\bu}_{k}
			&\le D_5,
			\label{eq:regularity_assumption}
		\end{align}
		for certain positive constants $D_i$, $i=1,\ldots,5$, which are assumed to be sufficiently small.
	\end{assumption}
	
	As stated in Assumption~\ref{ass:regular_solution}, we assume that a solution to the problem exists. The existence theory for this type of problem has been studied extensively. We refer, in particular, to the recent work \cite{Bulicek2023}, where an existence framework applicable to a wide class of models, including those considered here, is developed.

	\begin{theorem}[Existence and optimal a priori error estimate]\label{thm:nonlinear_existence_error} 
		Suppose that Assumptions~\ref{ass:viscosity_lipschitz}, \ref{ass:viscosity_bound}, and \ref{ass:regular_solution} hold, and that $\bff\in L^1(0,T;H^{-1}(\Omega)^d)$. Then, if $\mu_\infty$ is sufficiently large and the constants $D_1,\ldots,D_5$ in Assumption~\ref{ass:regular_solution} are small enough, there exists a solution of \eqref{eq:vms_momentum}-\eqref{eq:vms_pressure_gradient_subscale} satisfying
		\begin{align}
			\sup_{0\le t\le T}\double{\bu-\bu_h}^2+\sup_{0\le t\le T}\double{\tilde\bu_1}^2+\sup_{0\le t \le T}\double{\tilde\bu_2}^2+\mu_\infty\int_0^T\double{\grad (\bu-\bu_h)}^2\textrm{d}t\le \bu_\star h^{2k},& \label{Error-Bound-1}\\
			\int_0^T\tau_1^{-1}\double{\tilde\bu_2}^2\textrm{d}t+\int_0^T \tau_1\double{\grad(p-p_h)-\partial_t\tilde\bu_2}^2\le p_\star h^{2k},\label{Error-Bound-2}
		\end{align}
		where $\bu_\star,p_\star$ are appropriate dimensional factors that render the estimates dimensionally consistent, and independent of $h$.
	\end{theorem}
	\begin{proof} 
		We proceed according to the four steps described at the beginning of this section.
		
		\textbf{Step 1:}  \textit{The iterative mapping.} First we define the following iterative mapping: \\
		Let $\LW_h=L^2(0,T;\LV_h)\times L^2(0,T;Q_h)\times L^2(0,T;\tilde\BV_1)\times L^2(0,T;\tilde\BV_2)$, then
		$$\delta_h:\LW_h\to \LW_h,\hspace{0.5cm}(\hat\bu_h,\hat p_h,\hat\bu_1,\hat\bu_2)\mapsto(\bu_h,p_h,\tilde\bu_1,\tilde\bu_2),$$
		where $\BU_h=(\bu_h,p_h,\tilde\bu_1,\tilde\bu_2)$ satisfies \eqref{eq:linear_momentum}-\eqref{eq:linear_pressure_gradient_subscale}, for all $\BV_h=(\bv_h,q_h,\tilde\bv_1,\tilde\bv_2)\in\LX_h$ a.e. in $(0,T)$. It is clear that a fixed-point is a solution of the approximation system (\ref{eq:vms_momentum})-(\ref{eq:vms_pressure_gradient_subscale}). 
		
		\textbf{Step 2:} \textit{The iterative map is well defined and bounded on bounded sets. }The existence and uniqueness
		result proved in Subsection~\ref{subsec:linear_wellposedness}, shows that the iterative map $\delta_h$ is well-defined.
		In addition, the stability result, proved in Subsection~\ref{subsec:linear_stability}, ensures that the linearised problem is stable and {proves 
			a uniform bound on the solution in terms of the data only}. This step is crucial in the definition of the fixed-point mapping and implies, in particular, that the mapping $\delta_h$ is bounded on bounded sets.
		
		\textbf{Step 3:} \textit{Existence of an invariant ball.} We begin by defining an invariant ball. Let $R=h^k$, and for $\BV=(\bv,q,\tilde\bv_1,\tilde\bv_2)\in\LW_h$ we define the norm:
		\begin{equation}
			\double{\BV}_{B}^2:=\sup_{0\le t\le T}\double{\bv}^2+\int_0^T\left\{\mu_\infty\double{\grad\bv}^2+\tau_1^{-1}\double{\tilde\bv_1}^2+\frac{\tau_1^{-1}}{2}\double{\tilde\bv_2}^2\right\}\mathrm{d}t+\int_0^T\frac{\tau_1}{2}\double{\grad q-\partial_t\tilde\bv_2}^2\mathrm{d}t. \label{Norma-B}
		\end{equation}

		Let us now define the ball $\LB$ as follows,
		\begin{equation}
			\LB=\left\{\BV_h=(\bv_h,q_h,\tilde\bv_1,\tilde\bv_2):(0,T)\to\LW_h\,\,\,\text{such that }\double{\BV_h-\BU}_{B}\le\BU_\star R\right\}
		\end{equation}
		where {$\BU:=(\bu,p,0,0)$ is a solution of the continuous solution problem \eqref{eq:strong_momentum}-\eqref{eq:strong_initial}}, and $\BU_\star^2:=\bu^2_\star+p^2_\star$ will be built in the course of the proof.

		Let now $\hat{\BU}_h=(\hat{\bu}_h,\hat{p}_h,\hat\bu_1,\hat\bu_2)\in\LB$ be arbitrary, and let $(\bu_h,p_h,\tilde\bu_1,\tilde\bu_2)=\delta_h(\hat{\bu}_h\hat p_h,\hat\bu_1,\hat\bu_2)$ satisfying \eqref{eq:linear_momentum}-\eqref{eq:linear_pressure_gradient_subscale}. Then, subtracting the continuous and discrete problems we derive the following error equation:
		\begin{align}
			(\partial_t(\bu-\bu_h),\bv_h)+[\overline c(\bu,\bu,\bv_h)-\overline c(\hat\bu_h,\bu_h,\bv_h)] \nonumber\\
			+2(\mu(|\grad^s\bu|)\grad^s\bu-\mu(|\grad^s\hat\bu_h|)\grad^s\bu_h,\grad^s\bv_h)-(p-p_h,\div\bv_h)+(\tilde\bu_1,\mathcal{N}(\hat\bu_h,\bv_h))\nonumber\\
			-\tau_2(P_p^\perp[\div\bu_h],\div\bv_h)+(\div(\bu-\bu_h),q_h)+(P^\perp[\grad q_h],\tilde\bu_2)-(\partial_t\tilde\bu_1,\tilde\bv_1) \nonumber\\
			-\tau_1^{-1}(\tilde\bu_1,\tilde\bv_1)-\overline{c}(\hat\bu_h,\bu_h,\tilde\bv_1)-(\partial_t\tilde\bu_2,\tilde\bv_2)-\tau_1^{-1}(\tilde\bu_2,\tilde\bv_2)-(\grad p_h,\tilde\bv_2)&=0, \label{eq:Error}
		\end{align}
		for every test function $\BV_h=(\bv_h,q_h,\tilde\bv_1,\tilde\bv_2)$.
		
		Let us now define the approximation and interpolation errors:
		\begin{eqnarray}
			\Lambda = \bu-\mathcal{U}&\text{ and }&E=\mathcal{U}-\bu_h, \nonumber\\
			\Pi = p-\mathcal{P}&\text{ and }&G=\mathcal{P}-p_h
		\end{eqnarray}
		where $\mathcal{U}=P_{u,0}[\bu]$, $\mathcal{P}=P_p[p]$. In addition we define $e_u=\bu-\bu_h=\Lambda+E$ and $e_p=p-p_h=\Pi+G$.
		
		The proof will consist in proving that there exists a  $\BU_h$, solution  of \eqref{eq:vms_momentum}-\eqref{eq:vms_pressure_gradient_subscale}, that belongs to the ball $\LB$. Let $P[\BU]=(\mathcal{U},\mathcal{P},\boldsymbol{0,\boldsymbol{0}})$. By \eqref{eq:Error} we shall prove that
		\begin{equation}
			\double{P[\BU]-\BU_h}^2_{B}\le\varphi_1(D)\double{\BU-\hat\BU_h}_{B}^2+\double{\BU-P[\BU]}_I^2,\label{eq:desigualdadobjetivo}
		\end{equation}
		where $\varphi_1(D)$ is a certain polynomial function written in terms of the components of the vector of constants $D=(D_1,...,D_5)$ introduced in Assumption~\ref{ass:regular_solution} and the norm $\double{\cdot}_I$ is associated with the interpolation error and will be built over the proof. Thanks to the properties of interpolation, we shall check that
		\begin{equation}
			\double{\BU-P[\BU]}_I^2\le\varphi_2(D)h^{2k}, \label{eq:errorInterpolacion1}
		\end{equation}
		and  that
		$$\double{\BU-P[\BU]}_{B}^2\le\varphi_3(D)h^{2k}.$$
		Using the triangle inequality, the previous inequalities, and the fact that $\hat\BU_h\in\LB$, we have that
		\begin{eqnarray}
			\double{\BU-\BU_h}_{B}^2&\le&\double{P[\BU]-\BU_h}_{B}^2+\double{\BU-P[\BU]}_{B}^2 \nonumber\\
			&\le& \varphi_1(D)\BU_\star^2h^{2k}+\varphi_2(D)h^{2k}+\varphi_3(D)h^{2k}.\nonumber
		\end{eqnarray}
		By taking the components of $D$ small enough, we will be able to guarantee that $\double{\BU-\BU_h}^2_{B}
		\le \BU^2_\star h^{2k}$. This means that $\BU_h\in\mathcal{B}_h$. Therefore the goal is to prove \eqref{eq:desigualdadobjetivo} and check that \eqref{eq:errorInterpolacion1} holds. The constant $\BU_\star$ can be taken as $\BU_\star^2=c(\varphi_2(D)+\varphi_3(D))$, with $c>1$ whenever $D$ is such that $\varphi_1(D)+c^{-1}\le1$.
		
		In order to motivate/explain the appearance of the anisotropic norm, we start noticing that \eqref{eq:pressure_gradient_subscale_pointwise} leads to
		\begin{align}
			\tau_1^{-1}(\tilde\bu_2,\tilde\bu_2) &=\frac{1}{2}\tau_1^{-1}\double{\tilde\bu_2}^2+\frac{1}{2}\tau_1^{-1}\double{\tilde\bu_2}^2 \nonumber\\
			&= \frac{1}{2}\tau_1^{-1}\double{\tilde\bu_2}^2+\frac{1}{2}\left(\tau_1^{-1}\double{\tilde\bu_2+\tau_1P_u^\perp[\grad\mathcal{P}]}^2-\tau_1^{-1}2(\tilde\bu_2,\tau_1P_u^\perp[\grad\mathcal{P}])\right. \nonumber\\
			&\quad\left.-\tau_1^{-1} \double{\tau_1P_u^\perp[\grad\mathcal{P}]}^2\right)\nonumber\\
			&=\frac{1}{2}\tau_1^{-1}\double{\tilde\bu_2}^2+\frac{1}{2}\left(\tau_1\double{\partial_t\tilde\bu_2-P_u^\perp[\grad p_h]+P_u^\perp[\grad \LP]}^2\right)-R_9(t) \nonumber\\
			&= \frac{1}{2}\tau_1^{-1}\double{\tilde\bu_2}^2+\frac{1}{2}\left(\tau_1\double{\partial_t\tilde\bu_2-P_u^\perp[\grad G]}^2\right)-R_9(t), 
		\end{align}
		where $R_9(t)$ will be defined later.
		
		Then, we test  the error equation \eqref{eq:Error} with 
		$$\bv_h=E+\gamma\tau_1P_{u,0}[\grad G],\quad q_h=G,\quad\tilde\bv_1=-\tilde\bu_1,\quad\tilde\bv_2=-\tilde\bu_2,$$
		where $\gamma>0$ will be chosen small enough at a later stage,  and obtain
		\begin{align}
			\frac{1}{2}\frac{\textrm{d}}{\textrm{d}t}\left(\double{E}^2+\double{\tilde\bu_1}^2+\double{\tilde\bu_2}^2\right)+Q(t)=N(t)+\sum_{j=1}^{15}R_j(t), \label{eq:error_equality}
		\end{align}
		where
		\begin{align}
			Q(t)=&\,\,2\mu_\infty\double{\grad^sE}^2+\tau_1^{-1}\double{\tilde\bu_1}^2+\frac{\tau_1^{-1}}{2}\double{\tilde\bu_2}^2+\tau_1\left(\frac{1}{2}\double{P_u^\perp[\grad G]-\partial_t\tilde\bu_2}^2+\gamma\double{P_{u,0}[\grad G]}^2\right), \nonumber\\
			N(t)=&-\left\{\overline{c}(\bu,\bu,E)-\overline{c}(\hat\bu_h,\bu_h,E)\right\}-\gamma\tau_1\left\{\overline{c}(\bu,\bu,P_{u,0}[\grad G])-\overline{c}(\hat\bu_h,\bu_h,P_{u,0}[\grad G])\right\}, \nonumber
		\end{align}
		and the remainders $R_j(t)$ will be defined and bounded  below. 
		
		Note that from the stability result \eqref{eq:interpolation_condition}, we have that
		\begin{align}
			\double{\grad G-\partial_t\tilde\bu_2}^2\lesssim&\,\,\double{P_{u,0}[\grad G-\partial_t\tilde\bu_2]}^2+\double{P_u^\perp[\grad G-\partial_t\tilde\bu_2]}^2 \nonumber\\
			=&\,\,\double{P_{u,0}[\grad G]}^2+\double{P_u^\perp[\grad G]-\partial_t\tilde\bu_2}^2. \label{eq:pressure_error_control}
		\end{align}
		Then, using \eqref{eq:error_equality},  \eqref{eq:pressure_error_control}, and the definition \eqref{Norma-B} of the norm $\double{\cdot}_B$  we obtain
		\begin{align}
			\double{P[\BU]-\BU_h}_B^2\lesssim\sup_{0\le t\le T}\double{E}^2+\int_0^TQ(t)\,\,\mathrm{d}t\lesssim\int_0^T\left(N(t)+\sum_{j=1}^{15}R_j(t)\right)\mathrm{d}t. \nonumber
		\end{align}
		The objective is to see that all these terms in the RHS can be bounded as indicated in \eqref{eq:desigualdadobjetivo} or absorbed by the LHS. We do not explicitly specify the values of the dimensionless constants that appear in this process, denoted by $\varepsilon$ and $\xi$, and assume that they are chosen sufficiently small. 
		
		We now start bounding each of the $R_j(t)$. Using the $L^2(\Omega)$-orthogonality between $\Lambda$ and $E$, Young's inequality, and $|\div \bv|\le \sqrt{d}|\grad^s\bv|$, we have that
		\begin{align}
			R_1(t)&=(\partial_t\Lambda,E)=0, \nonumber\\
			R_2(t)&=(\Pi,\div E)\le \mu_\infty\varepsilon\double{\grad^s E}^2+\frac{d}{2\mu_\infty\varepsilon}\double{\Pi}^2. \nonumber
		\end{align}
		
		The term involving the discrete pressure error requires special treatment. Using (\ref{eq:interpolation_condition}) and the inequality $(a+b)^2\le 2a^2+2b^2$, we have,
		\begin{align}
			R_3(t)&= -(\div\Lambda, G)=(\Lambda,-\partial_t\tilde\bu_2+\grad G)+(\Lambda,\partial_t\tilde\bu_2) \nonumber\\
			&\le \tau_1^{-1}\frac{1}{2\varepsilon}\double{\Lambda}^2+\frac{\varepsilon\tau_1}{\beta_0^2}\left(\double{P_u^\perp[\partial_t\tilde\bu_2-\grad G]}^2+\double{P_{u,0}[\partial_t\tilde\bu_2-\grad G]}^2\right) +(\Lambda,\partial_t\tilde\bu_2).\nonumber
		\end{align}
		Now, for the last term above we use Young's inequality to get to
		\begin{align}
			(\Lambda,\partial_t\tilde\bu_2)&\le \frac{1}{2\tau_1\gamma}\double{\Lambda}^2+\frac{\gamma\tau_1}{2}\double{\partial_t\tilde\bu_2}^2. \label{derivada subescala2}
		\end{align}
		The term $S(t)=\frac{\gamma\tau_1}{2}\double{\partial_t\tilde\bu_2}^2$ will be analysed when bounding the term $R_{14}(t)$ below.
		
		The following term only needs Cauchy--Schwarz and Young's inequalities
		\begin{align}
			R_4(t)&=-2\mu_\infty(\grad^s\Lambda,\grad^sE)\le \mu_\infty\varepsilon\double{\grad^sE}^2+\frac{\mu_\infty}{\varepsilon}\double{\grad^s\Lambda}^2.
		\end{align}
		
		Recalling the decomposition of the apparent viscosity \eqref{eq:mu_phi_decomposition}, the following term corresponds to the non-linear part of the viscous term:
		\begin{align}
			R_5(t)&=-2(\phi(|\grad^s\bu|)\grad^s\bu-\phi(|\grad^s\hat\bu_h|)\grad^s\bu_h,\grad^s E) \nonumber\\
			&= -2(\left\{\phi(|\grad^s\bu|)-\phi(|\grad^s\hat\bu_h|)\right\}\grad^s\bu,\grad^s E)-2(\phi(|\grad^s\hat\bu_h|)\grad^s(\bu-\bu_h),\grad^s E). \nonumber
		\end{align}
		Using the Lipschitz continuity of $\phi$ stated in Assumption \ref{ass:viscosity_lipschitz}, together with the Cauchy--Schwarz' and Young's inequalities, and the bounds for $\phi$ given in \eqref{eq:phi_bounds}, we obtain
		\begin{align}
			R_5(t)&\le 2C_\phi D_2\double{\grad^s(\bu-\hat\bu_h)}\double{\grad^s E}\underbrace{-2(\phi(|\grad^s\hat\bu_h|)\grad^sE,\grad^sE)}_{\le 0}-2(\phi(|\grad^s\hat\bu_h|)\grad^s\Lambda,\grad^sE)  \nonumber\\
			&\le C\left\{\frac{D_2}{\mu_\infty\varepsilon}\double{\grad^s(\bu-\hat\bu_h)}^2+D_2\varepsilon\mu_\infty\double{\grad^sE}^2\right\}+2(\mu_0-\mu_\infty)\double{\grad^s\Lambda}\double{\grad^sE}. \label{eq:R4}
		\end{align}
		We can note that the first term of \eqref{eq:R4} falls in the structure \eqref{eq:desigualdadobjetivo}. For the second one we use Young's inequality as follows:
		\begin{align}
			\double{\grad^s\Lambda}\double{\grad^sE}\lesssim\mu_\infty\varepsilon\double{\grad^s E}^2+\frac{1}{\mu_\infty\varepsilon}\double{ \grad^s\Lambda}^2.
		\end{align}

		Taking the parameter $\varepsilon$ sufficiently small, it can be shown that the terms appearing in the bounds of  $R_1,...,R_5$ can either be absorbed by $Q(t)$ on the LHS of \eqref{eq:error_equality},  or are interpolation errors that behave as indicated in (\ref{eq:errorInterpolacion1}).
		
		Let us now discuss the right-hand-side terms associated with the stabilisation. Using Cauchy--Schwarz and Young's inequalities, the following terms can be easily controlled:
		\begin{align}
			R_6(t) &= (\tilde\bu_1,P_u^\perp[\mathcal{N}(\hat\bu_h,\mathcal{U})])\le\varepsilon\tau_1^{-1}\double{\tilde\bu_1}^2+\frac{\tau_1}{\varepsilon}\double{P_u^\perp[\mathcal{N}(\hat\bu_h,\mathcal{U})]}^2,\nonumber\\
			R_7(t) &= \tau_2(P_p^\perp[\div\mathcal{U}],\div E), \nonumber\\
			R_8(t) &= (\tilde\bu_2,P_u^\perp[\grad\mathcal{P}])\lesssim \varepsilon\tau_1^{-1}\double{\tilde\bu_2}^2+\frac{\tau_1}{\varepsilon}\double{P_u^\perp[\grad\mathcal{P}]}^2, \nonumber\\
			R_9(t)&= 2(\tilde\bu_2,P_u^\perp[\grad\mathcal{P}])+\tau_1\double{P_u^\perp[\grad\mathcal{P}]}^2.
		\end{align}
		Some of these terms need further treatment. Let us first consider $R_6(t)$,
		\begin{align}
			\tau_1\double{P_u^\perp[\mathcal{N}(\hat\bu_h,\mathcal{U})]}^2 \lesssim&\,\,
			\tau_1\double{P_u^\perp[\mathcal{N}(\hat\bu_h-\bu,\mathcal{U}-\bu)]}^2+\tau_1\double{P_u^\perp[\mathcal{N}(\bu,\mathcal{U}-\bu)]}^2\nonumber\\
			&+\tau_1\double{P_u^\perp[\mathcal{N}(\bu,\bu)]}^2+\tau_1\double{P_u^\perp[\mathcal{N}(\hat u_h-\bu,\bu)]}^2\nonumber
			\\
			\lesssim&\,\,\tau_1\double{P_u^\perp[(\hat\bu_h-\bu)\cdot\grad(\mathcal{U}-\bu)]}^2+\tau_1\double{P_u^\perp[\bu\cdot\grad(\mathcal{U}-\bu)]}^2 \nonumber\\
			&+\tau_1\double{P_u^\perp[\bu\cdot\grad\bu]}^2+ \tau_1\double{P_u^\perp[(\hat\bu_h-\bu)\cdot\grad\bu]}^2 \nonumber \\
			&+\tau_1\double{P_u^\perp[(\div(\hat\bu_h- \bu))\cdot(\mathcal{U}-\bu)]}^2+\tau_1\double{P_u^\perp[\div(\hat\bu_h-\bu)\cdot\bu]}^2,
		\end{align}
		where the six terms on the RHS can be shown to have the structure given by \eqref{eq:desigualdadobjetivo}. For example, using the inequality \eqref{eq:projection_W1infty}, we get
		\begin{align*}
			\tau_1
			\left\|P_u^\perp\!\left[(\hat{\bu}_h-\bu)\cdot\nabla(\bu-\LU)\right]\right\|^2
			&\lesssim
			\tau_1
			\|\hat{\bu}_h-\bu\|^2
			\|\nabla(\bu-\LU)\|_\infty^2
			\\
			&\lesssim
			\tau_1 \|\hat{\bu}_h-\bu\|^2
			\|\bu\|_{1,\infty}^2.
		\end{align*}
		The terms $R_7(t)$, $R_8(t)$ and $R_9(t)$ can be bounded by terms that can be absorbed by $Q(t)$ as well as by terms involving the orthogonal projection of an operator applied to the projection onto the finite element space of the continuous solution. For example, for $R_7(t)$
		$$ \double{P_p^\perp[\div \mathcal{U}]}^2\lesssim\double{P_p^\perp[\div(\mathcal{U}-\bu)]}^2+\double{P_p^\perp[\div\bu]}^2\lesssim\double{\bu}_{k+1}^2h^{2k}. $$

		The remaining terms are multiplied by $\gamma$. The first of these can be bounded by the Cauchy--Schwarz and Young's inequalities, as well as the inverse inequality
		\begin{eqnarray}
			R_{10}(t)&=& \gamma \tau_1(\Pi,\div P_{u,0}[\grad G])\le \gamma\tau_1\double{\grad\Pi}\cdot\double{P_{u,0}[\grad G]} \nonumber\\
			&\le& \frac{\gamma}{2\mu_\infty\xi_1}\double{\Pi}^2+\tau_1^2\frac{\gamma \mu_\infty c_{inv}\xi_1}{2h^2}\double{P_{u,0}[\grad G]}^2. \nonumber
		\end{eqnarray}
		The following term also requires special treatment:
		\begin{align}
			R_{11}(t)=&\,\, -2\gamma\tau_1\mu_\infty(\grad^s(\bu-\bu_h),\grad^s(P_{u,0}[\grad G])) \nonumber\\
			\lesssim&\,\, \gamma\tau_1 \mu_\infty\left(\double{\grad^sE}+\double{\grad^s\Lambda}\right)\double{\grad^s(P_{u,0}[\grad G])}. \label{eq:R10}
		\end{align}
		For the last term, we use inverse inequality and the expression of the stabilised coefficient $\tau_1$. Now for the first term of \eqref{eq:R10}, we use Young's and inverse inequalities to get
		\begin{align}
			\gamma\tau_1\mu_\infty\double{\grad^sE}\double{\grad^s(P_{u,0}[\grad G])}\lesssim \xi\mu_\infty\double{\grad^sE}^2+\tau_1\gamma\frac{\tau_1\mu_\infty\gamma}{h^2\xi}\double{P_{u,0}[\grad G]}^2. \nonumber
		\end{align}
		To see that the bounds obtained fall within the structure \eqref{eq:desigualdadobjetivo}, we have to use the definition of the stabilisation coefficient $\tau_1$ to check that $\mu_\infty\tau_1h^{-2}\lesssim c_1^{-1}$ and take $\gamma$ small enough to balance $\xi^{-1}$, which can be large.
		
		The next term is associated with the non-linear part of the diffusive term:
		\begin{align}
			R_{12}(t)=&\,\,-\gamma\tau_1(\phi(|\grad^s\bu|)\grad^s\bu-\phi(|\grad^s\hat\bu_h|)\grad^s\bu_h,\grad^s P_{u,0}[\grad G]) \nonumber\\
			=&\,\,-\gamma\tau_1(\{\phi(|\grad^s\bu|)-\phi(|\grad^s\hat\bu_h|)\}\grad^s\bu,\grad^sP_{u,0}[\grad G])\nonumber\\
			&\,\,+\gamma\tau_1(\phi(|\grad^s\hat\bu_h|)\grad^s(\bu-\bu_h),\grad^sP_{u,0}[\grad G]). \nonumber
		\end{align}
		All these terms can be bounded using a similar strategy as for $R_5(t)$ and using the expression of the stabilisation coefficient $\tau_1$. In this case we take $\gamma$ small enough to balance $\xi^{-1}$, as before.

		The following terms to consider come from stabilisation
		\begin{align*}
			R_{13}(t)&=\gamma\tau_1(\tilde\bu_1,\mathcal{N}(\hat\bu_h,P_{u,0}[\grad G])),\\
			R_{14}(t)&= \gamma\tau_1\tau_2(P_p^\perp[\div(\bu_h-\mathcal{U}+\mathcal{U})],\div P_{u,0}[\grad G]).
		\end{align*}
		These terms can be bounded using arguments analogous to those already used.
		
		The next term requires some elaboration. Observe that,
		$$(\partial_t(\bu-\bu_h),P_{u,0}[\grad G])=(\partial_t\bu,P_{u,0}[\grad G])+(\partial_t\div\bu_h,G).$$
		Since $\bu$ is divergence-free and vanishes on the boundary, we have
		\begin{align}
			\gamma\tau_1(\partial_t\bu,P_{u,0}[\grad G])&=\gamma\tau_1(\partial_t\bu,P_{u,0}[\grad G]-\grad G)=-\gamma\tau_1\left(P_{u,0}^\perp[\partial_t\bu],P_{u,0}^\perp[\grad G]\right) \nonumber\\
			&\le \gamma\tau_1\double{P_{u,0}^\perp[\partial_t\bu]}\double{P_{u,0}^\perp[\grad G]}. \nonumber
		\end{align}
		The stability condition \eqref{eq:interpolation_condition} implies that $\double{P_{u,0}^\perp[\grad G]}\lesssim \double{P_{u,0}[\grad G]}+\double{P_u^\perp[\grad G]}$, then
		\begin{align}
			\gamma\tau_1(\partial_t\bu,P_{u,0}[\grad G])\le&\,\,\gamma\tau_1 \double{P_{u,0}^\perp[\partial_t\bu]}\double{P_{u,0}[\grad G]}+\gamma\tau_1 \double{P_{u,0}^\perp[\partial_t\bu]}\double{P_{u}^\perp[\grad G]} \nonumber\\
			=&\,\, \frac{\tau_1\gamma}{\xi}\double{P_{u,0}^\perp[\partial_t\bu]}^2+\tau_1\gamma\xi\double{P_{u,0}[\partial_t\tilde\bu_2-\grad G]}^2 \nonumber\\
			&+\frac{\tau_1\gamma}{2}\double{P_{u,0}^\perp[\partial_t\bu]}^2+\frac{\tau_1\gamma}{2}\double{P_{u}^\perp[\grad G]}^2. \label{iq:1}
		\end{align}
		
		On the other hand, differentiating the continuity equation of the stabilised method with respect to time, we can write
		\begin{align}
			(\partial_t\div\bu_h, G)=\left(\partial_t\tilde\bu_2,P_u^\perp[\grad G]\right). \label{eq:2}
		\end{align}
		From these two inequalities \eqref{iq:1} and \eqref{eq:2} we obtain the following
		\begin{align}
			R_{15}(t)=&\,\, -\gamma\tau_1(\partial_t(\bu-\bu_h),P_{u,0}[\grad G]) \nonumber\\
			=&\,\, -\gamma\tau_1(P_{u,0}^\perp[\partial_t \bu],P_{u,0}^\perp[\grad G])-\gamma\tau_1(\partial_t\tilde\bu_2,P_u^\perp[\grad G]) \nonumber\\
			\lesssim&\,\,\frac{\gamma\tau_1}{\xi}\double{P_{u,0}^\perp[\partial_t\bu]}^2+\gamma\tau_1\xi\double{P_{u,0}[\partial_t\tilde\bu_2-\grad G]}^2\nonumber\\
			&\,\, +\frac{\gamma\tau_1}{2}\double{P_u^\perp[\partial_t\bu]}^2+\frac{\gamma\tau_1}{2}\double{P_u^\perp[\grad G]}^2-\gamma\tau_1(\partial_t\tilde\bu_2,P_u^\perp[\grad G]).
		\end{align}
		Now to have control in the last two terms we recover $S(t)$  from $R_3(t)$. In fact,
		\begin{align}
			\frac{\gamma\tau_1}{2}\double{P_u^\perp[\grad G]}^2-&\gamma\tau_1(\partial_t\tilde\bu_2,P_u^\perp[\grad G])+S(t)\nonumber\\ &=\frac{\gamma\tau_1}{2}\left(\double{P_u^\perp[\grad G]}^2-2(\partial_t\tilde\bu_2,P_u^\perp[\grad G])+\double{\partial_t\tilde\bu_2}^2\right) \nonumber\\
			&=\frac{\gamma\tau_1}{2}\double{\partial_t\tilde\bu_2-P_{u}^\perp[\grad G]}^2.\label{eq:semidiscrete_anisotropy}
		\end{align}
		Then, choosing $\gamma$ sufficiently small, we conclude that this term can be absorbed by the LHS with $Q(t)$. 
		
		It only remains to deal with the convective term $\mathcal{N}(t)$. The first two components of
		$\mathcal{N}(t)$ can be written as follows:
		\begin{align}
			& \overline{c}(\bu,\bu,\bv_h)-\overline{c}(\hat\bu_h,\bu_h,E) 
			= \overline{c}(\bu-\hat\bu_h,\bu,E)-\overline{c}(\bu-\hat\bu_h,\Lambda,E)+\overline{c}(\bu,\Lambda,E).
		\end{align}
		The three terms on the right-hand side above can be easily bounded so that they fit the structure \eqref{eq:desigualdadobjetivo}.
		
		Finally, the remaining two parts of the convective term can be written as:
		\begin{align}
			& \gamma\tau_1\left\{\overline{c}(\bu,\bu,P_{u,0}[\grad G])-\overline{c}(\hat\bu_h,\bu_h,P_{u,0}[\grad G])\right\}\nonumber\\
			&\hspace{0.5cm}= \gamma\tau_1\left\{\overline{c}(\bu-\hat \bu_h,\bu,P_{u,0}[\grad G])-\overline{c}(\bu-\hat\bu_h,\bu-\bu_h,P_{u,0}[\grad G])
			+\overline{c}(\bu,\bu-\bu_h,P_{u,0}[\grad G])\right\}.
		\end{align}
		Each term in the right-hand side above can be bounded following the same strategy as used previously. 
		
		This concludes the proof of the third step.
		
		\textbf{Step 4.} \textit{Fixed-point theorem. }According to step 3, the ball $\LB_h$ is invariant under the map $\delta_h$, that is, $\delta_h(\LB_h)\subset\LB_h$. Therefore, by applying {Brouwer's} fixed-point theorem, we conclude that there exists $\BU_h=(\bu_h,p_h,\tilde\bu_1,\tilde\bu_2)\in \LB_h$ such that $\delta_h(\BU_h)=\BU_h$ and hence is a solution of \eqref{eq:vms_momentum}-\eqref{eq:vms_pressure_gradient_subscale}. From the definition of the ball $\LB_h$, we also obtain the error estimates \eqref{Error-Bound-1} and \eqref{Error-Bound-2}.
	\end{proof}

	\begin{remark}\label{rem:dynamic_vs_quasistatic}
		The dynamic pressure-related subscale is a structural ingredient of the analysis. It provides the time-derivative contribution needed to control the projected pressure-gradient error through
		$$\partial_t\tilde\bu_2-P_u^\perp[\grad G].$$
		For quasi-static subscales, the term $\partial_t\tilde\bu_2$ is absent, and this mechanism cannot absorb the same contribution. In that case, one typically needs an additional large-time assumption; see \cite[Theorem 6.2]{BarrenecheaCastillo2019}. Thus, the dynamic pressure-gradient subscale is not used only to improve transient accuracy; it is required to close the semi-discrete stability mechanism without imposing an additional restriction of this type.
	\end{remark}

	\section{Weak-in-time control of the pressure-convection balance}
	\label{sec:pressure_convection_estimate}
	
	The stability norm used above controls the pressure-gradient contribution through the dynamic subscale equation. The next result makes this control explicit at the level of the discrete momentum balance. It is not an error estimate, and therefore it does not affect the optimal-order bounds obtained in Theorem~\ref{thm:nonlinear_existence_error}; rather, it identifies an additional weak-in-time stability property of the method and clarifies what the non-residual separation achieves at the level of the unresolved momentum balance. More precisely, we prove a bound for
	\[
	\grad p_h+\mathcal N(\bu_h,\bu_h),
	\]
	which is the pressure-convection balance entering the unresolved momentum equation. While this bound is relatively easy to obtain for the linearised problem, to our knowledge it is much scarcer for non-linear problems. Related bounds are available for Newtonian stabilised formulations; see, for example, \cite[Theorem~4.3]{Badia2010}. The estimate below extends this type of weak-in-time control to the present non-linear generalised Newtonian setting under the assumptions used in the preceding analysis. The trade-off of obtaining such a bound in the non-linear case is that the bound is proved in a rather weak norm. Specifically, this is the norm of the dual of $H^1_0(0,T;L^2(\Omega)^d)$, which we denote by $H^{-1}(0,T;L^2(\Omega)^d)$.

	Let $T>0$ be fixed. We introduce the time-independent parameter
	\begin{equation}
		\tau_0^{-1}
		=
		c_1 \frac{\mu_0}{h^2}
		+
		c_2 
		\frac{\sup_{t\in[0,T]}\double{\bu_h(t)}_{\infty}}{h}.
		\label{eq:tau0_definition}
	\end{equation}
	For fixed $h>0$, this parameter is well defined. Indeed, the inverse inequality and the stability estimate imply that
	\[
	\sup_{t\in[0,T]}\double{\bu_h(t)}_{\infty}<\infty .
	\]
	Moreover, the definition of $\tau_0$ gives the scaling relations
	\begin{equation}
		\tau_0^{1/2}\frac{\mu_0}{h}
		\lesssim
		\mu_0^{1/2},
		\qquad
		\tau_0^{1/2}
		\frac{\double{\bu_h}_{\infty}}{h}
		\lesssim
		\tau_1^{-1/2},
		\qquad
		h^{-1}(\tau_1\tau_2)^{1/2}
		\lesssim 1.
		\label{eq:tau0_scalings}
	\end{equation}
	Here $\tau_1$ and $\tau_2$ are understood as the stabilisation parameters evaluated at the non-linear solution $\bu_h$.

	\begin{theorem}[Weak-in-time pressure-convection estimate]
		\label{thm:pressure_convection_estimate}
		Let $\Omega\subset\mathbb R^d$, $d=2,3$, and let
		$(\bu_h,p_h,\tilde\bu_1,\tilde\bu_2)$ be a solution of the non-linear VMS problem
		\eqref{eq:vms_momentum}-\eqref{eq:vms_pressure_gradient_subscale}. If $\bff\in L^2(0,T;H^{-1}(\Omega)^d)$, then
		\begin{equation}
			\tau_0^{1/2}
			\double{
				\grad p_h+\mathcal N(\bu_h,\bu_h)
			}_{H^{-1}(0,T;L^2(\Omega)^d)}
			\le C,
			\label{eq:pressure_convection_bound}
		\end{equation}
		where $C$ is independent of $h$.
	\end{theorem}
	
	\begin{proof} 
		Let $\bv\in H^1_0(0,T;L^2(\Omega)^d)$ and set
		\[
		\bv_h=P_{u,0}[\bv]\quad\text{\textrm{a.e. in~ } (0,T)}.
		\]
		We first estimate the resolved component of
		$\grad p_h+\mathcal N(\bu_h,\bu_h)$. Testing the momentum equation \eqref{eq:vms_momentum} with
		$\bv_h$ gives
		\begin{align}
			&\int_0^T
			\tau_0^{1/2}
			\left(
			P_{u,0}[
			\grad p_h+\mathcal N(\bu_h,\bu_h)
			],
			\bv_h
			\right)\,\textrm{d}t
			\nonumber\\
			&=
			-\int_0^T
			\tau_0^{1/2}
			\left\{
			(\partial_t\bu_h,\bv_h)
			+
			(2\mu(|\grad^s\bu_h|)\grad^s\bu_h,\grad^s\bv_h)
			\right.
			 \label{eq:resolved_pressure_convection_identity}\\
			&\hspace{2.25cm}\left.+
			(\tau_2\div\bv_h,P_p^\perp[\div\bu_h])-
			\overline c(\bu_h,\bv_h,\tilde\bu_1)
			\right\}\,\textrm{d}t+\int_0^T
			\tau_0^{1/2}
			\left<\bff,\bv_h\right>\,\mathrm{d}t.\nonumber
		\end{align}
		
		Since $\bv\in H^1_0(0,T;L^2(\Omega)^d)$, integration by parts in time yields
		\begin{align}
			-\int_0^T
			\tau_0^{1/2}
			(\partial_t\bu_h,\bv_h)\,\textrm{d}t
			&=
			\int_0^T
			\tau_0^{1/2}
			(\bu_h,\partial_t\bv_h)\,\textrm{d}t\nonumber\\
			&\lesssim
			\tau_0^{1/2}
			\left(
			\int_0^T\double{\bu_h}^2\,\textrm{d}t
			\right)^{1/2}
			\left(
			\int_0^T\double{\partial_t\bv_h}^2\,\textrm{d}t
			\right)^{1/2}.
			\label{eq:pc_time_term}
		\end{align}
		The convective subscale term is bounded by the inverse inequality and
		\eqref{eq:tau0_scalings}:
		\begin{align}
			\int_0^T
			\tau_0^{1/2}
			\left|
			\overline c(\bu_h,\bv_h,\tilde\bu_1)
			\right|\,\textrm{d}t
			&\lesssim
			\int_0^T
			\tau_0^{1/2}
			\double{\bu_h}_{\infty}
			h^{-1}
			\double{\bv_h}
			\double{\tilde\bu_1}\,\textrm{d}t
			\nonumber\\
			&\lesssim
			\left(
			\int_0^T
			\tau_1^{-1}\double{\tilde\bu_1}^2\,\textrm{d}t
			\right)^{1/2}
			\left(
			\int_0^T
			\double{\bv_h}^2\,\textrm{d}t
			\right)^{1/2}.
			\label{eq:pc_convective_term}
		\end{align}
		For the diffusive term, using $\mu\le\mu_0$, the inverse inequality and
		\eqref{eq:tau0_scalings}, we obtain
		\begin{align}
			\int_0^T
			\tau_0^{1/2}
			\left|
			(2\mu(\bx,|\grad^s\bu_h|)\grad^s\bu_h,\grad^s\bv_h)
			\right|\,\textrm{d}t
			&\lesssim
			\int_0^T
			\tau_0^{1/2}
			\mu_0 h^{-1}
			\double{\grad^s\bu_h}
			\double{\bv_h}\,\textrm{d}t
			\nonumber\\
			&\lesssim
			\left(
			\int_0^T
			\mu_0\double{\grad^s\bu_h}^2\,\textrm{d}t
			\right)^{1/2}
			\left(
			\int_0^T
			\double{\bv_h}^2\,\textrm{d}t
			\right)^{1/2}.
			\label{eq:pc_diffusive_term}
		\end{align}
		The pressure-divergence stabilisation term is bounded as
		\begin{align}
			\int_0^T
			\tau_0^{1/2}
			\left|
			\tau_2(\div\bv_h,P_p^\perp[\div\bu_h])
			\right|\,\textrm{d}t
			&\lesssim
			\int_0^T
			h^{-1}(\tau_1\tau_2)^{1/2}
			\tau_2^{1/2}
			\double{P_p^\perp[\div\bu_h]}
			\double{\bv_h}\,\textrm{d}t
			\nonumber\\
			&\lesssim
			\left(
			\int_0^T
			\tau_2
			\double{P_p^\perp[\div\bu_h]}^2\,\textrm{d}t
			\right)^{1/2}
			\left(
			\int_0^T
			\double{\bv_h}^2\,\textrm{d}t
			\right)^{1/2}.
			\label{eq:pc_divergence_term}
		\end{align}
		For the forcing term, the inverse inequality and the first relation in \eqref{eq:tau0_scalings} give
		\begin{align}
			\int_0^T
			\tau_0^{1/2}
			\left|\left<\bff,\bv_h\right>\right|\,\mathrm{d}t
			&\lesssim
			\int_0^T
			\tau_0^{1/2}
			\double{\bff}_{H^{-1}}
			\double{\bv_h}_{1}
			\,\mathrm{d}t
			\nonumber\\
			&\lesssim
			\left(
			\int_0^T
			\double{\bff}_{H^{-1}}^2\,\mathrm{d}t
			\right)^{1/2}
			\left(
			\int_0^T
			\double{\bv_h}^2\,\mathrm{d}t
			\right)^{1/2}.
			\label{eq:pc_force_term}
		\end{align}
		
		Combining
		\eqref{eq:pc_time_term}-\eqref{eq:pc_force_term}, using the stability estimate and the fact that $\|\bv_h\|\le \|\bv\|$ and $\|\partial_t\bv_h\|\le \|\partial_t\bv\|$ we obtain
		\begin{equation}
			\tau_0^{1/2}
			\double{
				P_{u,0}[
				\grad p_h+\mathcal N(\bu_h,\bu_h)
				]
			}_{H^{-1}(0,T;L^2(\Omega)^d)}
			\le C.
			\label{eq:pc_resolved_bound}
		\end{equation}

		It remains to estimate the unresolved component. Adding the two dynamic
		velocity subscale equations give
		\begin{equation}
			P_u^\perp[
			\grad p_h+\mathcal N(\bu_h,\bu_h)
			]
			=
			-
			\partial_t(\tilde\bu_1+\tilde\bu_2)
			-
			\tau_1^{-1}(\tilde\bu_1+\tilde\bu_2).
			\label{eq:unresolved_pressure_convection_identity}
		\end{equation}
		Testing this identity with $\bv\in H^1_0(0,T;L^2(\Omega)^d)$ and integrating in time yields
		\begin{align}
			&\int_0^T
			\tau_0^{1/2}
			\left(
			P_u^\perp[
			\grad p_h+\mathcal N(\bu_h,\bu_h)
			],
			\bv
			\right)\,\textrm{d}t
			\nonumber\\
			&=
			-\int_0^T
			\tau_0^{1/2}
			\left(
			\partial_t(\tilde\bu_1+\tilde\bu_2),
			\bv
			\right)\,\textrm{d}t
			-
			\int_0^T
			\tau_0^{1/2}
			\left(
			\tau_1^{-1}(\tilde\bu_1+\tilde\bu_2),
			\bv
			\right)\,\textrm{d}t .
			\label{eq:unresolved_pc_testing}
		\end{align}
		The first term on the right-hand side is integrated by parts in time:
		\begin{align}
			-\int_0^T
			\tau_0^{1/2}
			(
			\partial_t(\tilde\bu_1+\tilde\bu_2),
			\bv
			)\,\textrm{d}t
			&=
			\int_0^T
			\tau_0^{1/2}
			(
			\tilde\bu_1+\tilde\bu_2,
			\partial_t\bv
			)\,\textrm{d}t
			\nonumber\\
			&\lesssim
			\tau_0^{1/2}
			\left(
			\int_0^T
			\double{\tilde\bu_1+\tilde\bu_2}^2\,\textrm{d}t
			\right)^{1/2}
			\left(
			\int_0^T
			\double{\partial_t\bv}^2\,\textrm{d}t
			\right)^{1/2}.
			\label{eq:unresolved_pc_time}
		\end{align}
		For the second term, the definition of $\tau_0$ gives, up to constants independent of $h$,
		\[
		\tau_0^{1/2}\tau_1^{-1}
		\lesssim
		\tau_1^{-1/2}.
		\]
		Hence
		\begin{align}
			\int_0^T
			\tau_0^{1/2}
			\left|
			(
			\tau_1^{-1}(\tilde\bu_1+\tilde\bu_2),
			\bv
			)
			\right|\,\textrm{d}t
			&\lesssim
			\left(
			\int_0^T
			\tau_1^{-1}
			\double{\tilde\bu_1+\tilde\bu_2}^2\,\textrm{d}t
			\right)^{1/2}
			\left(
			\int_0^T
			\double{\bv}^2\,\textrm{d}t
			\right)^{1/2}.
			\label{eq:unresolved_pc_tau}
		\end{align}
		Using again the stability estimate, \eqref{eq:unresolved_pc_time} and
		\eqref{eq:unresolved_pc_tau} imply
		\begin{equation}
			\tau_0^{1/2}
			\double{
				P_u^\perp[
				\grad p_h+\mathcal N(\bu_h,\bu_h)
				]
			}_{H^{-1}(0,T;L^2(\Omega)^d)}
			\le C.
			\label{eq:pc_unresolved_bound}
		\end{equation}
		
		Finally, using compatibility assumption \eqref{eq:interpolation_condition} we get
		\[
		\|\grad p_h+\mathcal N(\bu_h,\bu_h)\|
		\lesssim 
		\|P_{u,0}[
		\grad p_h+\mathcal N(\bu_h,\bu_h)
		]\|
		+
		\|P_u^\perp[
		\grad p_h+\mathcal N(\bu_h,\bu_h)
		]\|,
		\]
		and combining \eqref{eq:pc_resolved_bound} with
		\eqref{eq:pc_unresolved_bound}, we obtain
		\eqref{eq:pressure_convection_bound}.
	\end{proof}
	
	\begin{remark}\label{rem:pressure_convection_meaning}
		Theorem~\ref{thm:pressure_convection_estimate} should be read as an additional stability result. Its role is to show that the non-residual VMS formulation controls the discrete pressure-convection balance in a weak-in-time sense. The resolved part of this balance is controlled through the momentum equation and the stability estimate, whereas the unresolved part is controlled directly by the dynamic subscale equations, which is consistent with the term-by-term structure of the formulation: the method not only stabilises the velocity and pressure variables, but also provides a controlled representation of the unresolved momentum balance.
	\end{remark}
	
	\section{Conclusions}\label{sec:conclusions}
	
	We have analysed a non-residual variational multiscale finite element formulation with dynamic subscales for incompressible generalised Newtonian Navier--Stokes flows. The analysis covers bounded and Lipschitz-continuous apparent viscosities, including several regularised rheological laws of physical interest, while preserving a framework suitable for finite element stability and error estimates.
	
	For the linearised semi-discrete problem, we proved well-posedness and unconditional stability in an anisotropic VMS norm, with no restrictions on the length of the time interval $(0,T)$. 
	For the non-linear formulation, a fixed-point argument yielded the existence of a discrete solution and optimal-order a priori estimates under suitable regularity and smallness assumptions. The estimates also identify where the dynamic nature of the subscales enters the proof, namely, in controlling the projected pressure-gradient contribution. An additional weak-in-time estimate was also obtained for the discrete pressure-convection balance.
	
	The dynamic subscales play a structural role in the analysis. In particular, the pressure-related dynamic subscale provides the control needed to handle the projected pressure-gradient error, avoiding the large-time restriction associated with quasi-static subscale formulations. The fully discrete analysis under anisotropic time-space discretisations remains a natural continuation of the present semi-discrete theory. A promising starting point in this direction is provided by the 
	implicit-explicit strategies recently analysed in \cite{BarrenecheaCastilloPacheco2024} 
	for variable-viscosity Navier--Stokes problems, which decouple the velocity components 
	by treating the coupling part of the viscous term explicitly while preserving 
	unconditional stability; combining such strategies with the present dynamic 
	non-residual VMS formulation could lead to a fully discrete, computationally 
	efficient scheme.
	
	\newpage
	\section*{Funding}
	
	The second and fourth authors acknowledge the support given by the Agencia Nacional de Investigación y Desarrollo (ANID) through project FONDECYT 1250287. The third author gratefully acknowledges the support received from the ICREA Acadèmia Programme of the Catalan Government. The second, fourth and fifth authors are grateful for the support of the Dirección de Investigación y Desarrollo (DICYT) under project DICYT Asociativo, code 042632GM\_DAS.

\bibliographystyle{abbrv}
\bibliography{References}

@article{Escobar2026,
    author = {Escobar, D. and Pacheco, D. R. Q. and Aguirre, A. and Castillo, E.},
    title = {A term-by-term variational multiscale method with dynamic subscales for incompressible turbulent aerodynamics},
    journal = {Physics of Fluids},
    volume = {38},
    number = {4},
    pages = {045148},
    year = {2026},
    month = {04},
    issn = {1070-6631},
    doi = {10.1063/5.0324342},
    url = {https://doi.org/10.1063/5.0324342},
    eprint = {https://pubs.aip.org/aip/pof/article-pdf/doi/10.1063/5.0324342/20982630/045148_1_5.0324342.pdf},
}

@article {GuerreroCastillo2025,
    AUTHOR = {Guerrero, F. and Castillo, E. and Galarce, F. and Pacheco, D.
              R. Q.},
     TITLE = {Spatially and temporally high-order dynamic nonlinear
              variational multiscale methods for generalized {N}ewtonian
              flows},
   JOURNAL = {Commun. Nonlinear Sci. Numer. Simul.},
  FJOURNAL = {Communications in Nonlinear Science and Numerical Simulation},
    VOLUME = {140},
      YEAR = {2025},
     PAGES = {Paper No. 108368, 21},
      ISSN = {1007-5704,1878-7274},
   MRCLASS = {65M60 (76M10)},
  MRNUMBER = {4801461},
       DOI = {10.1016/j.cnsns.2024.108368}
}

@article{FehnKronbichlerLube2025,
  author  = {Fehn, Niklas and Kronbichler, Martin and Lube, Gert},
  title   = {From Anomalous Dissipation Through {E}uler Singularities to Stabilized Finite Element Methods for Turbulent Flows},
  journal = {Flow Turbul. Combust.},
  year    = {2025},
  volume  = {115},
  number  = {1},
  pages   = {347--388},
  doi     = {10.1007/s10494-025-00639-6},
  publisher = {Springer},
}

@article {BarrenecheaCastilloPacheco2024,
    AUTHOR = {Barrenechea, Gabriel and Castillo, Ernesto and Pacheco,
              Douglas},
     TITLE = {Implicit-explicit schemes for incompressible flow problems
              with variable viscosity},
   JOURNAL = {SIAM J. Sci. Comput.},
  FJOURNAL = {SIAM Journal on Scientific Computing},
    VOLUME = {46},
      YEAR = {2024},
    NUMBER = {4},
     PAGES = {A2660--A2682}
}

@article {BarrenecheaSuli2023,
    AUTHOR = {Barrenechea, Gabriel R. and S\"uli, Endre},
     TITLE = {Analysis of a stabilised finite element method for power-law
              fluids},
   JOURNAL = {Constr. Approx.},
  FJOURNAL = {Constructive Approximation. An International Journal for
              Approximations and Expansions},
    VOLUME = {57},
      YEAR = {2023},
    NUMBER = {2},
     PAGES = {295--325},
       DOI = {10.1007/s00365-022-09591-4}
}

@article {PascalSuli2022,
    AUTHOR = {Heid, Pascal and S\"uli, Endre},
     TITLE = {On the convergence rate of the {K}a\v canov scheme for
              shear-thinning fluids},
   JOURNAL = {Calcolo},
  FJOURNAL = {Calcolo. A Quarterly on Numerical Analysis and Theory of
              Computation},
    VOLUME = {59},
      YEAR = {2022},
    NUMBER = {1},
     PAGES = {Paper No. 4, 27},
      ISSN = {0008-0624,1126-5434},
   MRCLASS = {35Q35 (76A05)},
  MRNUMBER = {4345847},
       DOI = {10.1007/s10092-021-00444-3}
}

@article {BerselliRusicka2021,
    AUTHOR = {Berselli, Luigi C. and R{\r u}\v{z}i\v{c}ka, Michael},
     TITLE = {Optimal error estimate for a space-time discretization for
              incompressible generalized {N}ewtonian fluids: the {D}irichlet
              problem},
   JOURNAL = {Partial Differ. Equ. Appl.},
  FJOURNAL = {Partial Differential Equations and Applications},
    VOLUME = {2},
      YEAR = {2021},
    NUMBER = {4},
     PAGES = {Paper No. 59, 23},
}

@article {Diening2021,
    AUTHOR = {Diening, Lars and Storn, Johannes and Tscherpel, Tabea},
     TITLE = {On the {S}obolev and {$L^p$}-stability of the
              {$L^2$}-projection},
   JOURNAL = {SIAM J. Numer. Anal.},
  FJOURNAL = {SIAM Journal on Numerical Analysis},
    VOLUME = {59},
      YEAR = {2021},
    NUMBER = {5},
     PAGES = {2571--2607},
      ISSN = {0036-1429,1095-7170},
   MRCLASS = {65N30 (46E35 65N12 65N50)},
  MRNUMBER = {4320894},
MRREVIEWER = {Nicolae\ Pop},
       DOI = {10.1137/20M1358013},
       URL = {https://doi.org/10.1137/20M1358013},
}

@article {Farrell2020,
    AUTHOR = {Farrell, P. E. and Gazca-Orozco, P. A. and S\"uli, E.},
     TITLE = {Numerical analysis of unsteady implicitly constituted
              incompressible fluids: 3-field formulation},
   JOURNAL = {SIAM J. Numer. Anal.},
  FJOURNAL = {SIAM J. Numer. Anal.},
    VOLUME = {58},
      YEAR = {2020},
    NUMBER = {1},
     PAGES = {757--787},
      ISSN = {0036-1429,1095-7170},
   MRCLASS = {65M60 (35Q35 65M12 76A05)},
  MRNUMBER = {4066569},
MRREVIEWER = {Mar\'ia\ Gonz\'alez Taboada},
       DOI = {10.1137/19M125738X},
       URL = {https://doi-org.ezproxy.usach.cl/10.1137/19M125738X}
}

@article {CastilloCodina2019,
    AUTHOR = {Castillo, E. and Codina, R.},
     TITLE = {Dynamic term-by-term stabilized finite element formulation
              using orthogonal subgrid-scales for the incompressible
              {N}avier-{S}tokes problem},
   JOURNAL = {Comput. Methods Appl. Mech. Engrg.},
  FJOURNAL = {Comput. Methods Appl. Mech. Engrg.},
    VOLUME = {349},
      YEAR = {2019},
     PAGES = {701--721},
}

@article{BarrenecheaCastillo2019,
    author = {Barrenechea, Gabriel R and Castillo, Ernesto and Codina, Ramon},
    title = {Time-dependent semidiscrete analysis of the viscoelastic fluid flow problem using a variational multiscale stabilized formulation},
    journal = {IMA J. Numer. Anal.},
    volume = {39},
    number = {2},
    pages = {792-819},
    year = {2019},
    month = {04},
    issn = {0272-4979},
    doi = {10.1093/imanum/dry018},
    url = {https://doi.org/10.1093/imanum/dry018},
    eprint = {https://academic.oup.com/imajna/article-pdf/39/2/792/28378226/dry018.pdf},
}

@article{SuliTabea2019,
    AUTHOR = {S\"uli, Endre and Tscherpel, Tabea},
     TITLE = {Fully discrete finite element approximation of unsteady flows
              of implicitly constituted incompressible fluids},
   JOURNAL = {IMA J. Numer. Anal.},
  FJOURNAL = {IMA Journal of Numerical Analysis},
    VOLUME = {40},
      YEAR = {2020},
    NUMBER = {2},
     PAGES = {801--849},
      ISSN = {0272-4979,1464-3642},
   MRCLASS = {65M60 (65M12 76A05 76Dxx)},
  MRNUMBER = {4092271},
MRREVIEWER = {Hans-Joachim\ Bungartz},
       DOI = {10.1093/imanum/dry097},
       URL = {https://doi-org.ezproxy.usach.cl/10.1093/imanum/dry097},
}

@article {Codina2018,
    AUTHOR = {Codina, Ramon},
     TITLE = {On {$hp$} convergence of stabilized finite element methods for
              the convection-diffusion equation},
   JOURNAL = {SeMA J.},
  FJOURNAL = {SeMA Journal. Boletin de la Sociedad Espa\~nola de
              Matem\'atica Aplicada},
    VOLUME = {75},
      YEAR = {2018},
    NUMBER = {4},
     PAGES = {591--606},
      ISSN = {2254-3902,2281-7875},
   MRCLASS = {65N30 (65N12)},
  MRNUMBER = {3875916},
       DOI = {10.1007/s40324-018-0154-4},
       URL = {https://doi-org.ezproxy.usach.cl/10.1007/s40324-018-0154-4}
}

@article{Castillo2017,
  author    = {Castillo, Ernesto and Codina, Ramon},
  title     = {Numerical analysis of a stabilized finite element
               approximation for the three-field linearized viscoelastic
               fluid problem using arbitrary interpolations},
  journal   = {ESAIM: Mathematical Modelling and Numerical Analysis},
  volume    = {51},
  number    = {4},
  pages     = {1407--1427},
  year      = {2017},
  doi       = {10.1051/m2an/2016068}
}

@article {Castillo2014,
    AUTHOR = {Castillo, Ernesto and Codina, Ramon},
     TITLE = {Variational multi-scale stabilized formulations for the
              stationary three-field incompressible viscoelastic flow
              problem},
   JOURNAL = {Comput. Methods Appl. Mech. Engrg.},
  FJOURNAL = {Computer Methods in Applied Mechanics and Engineering},
    VOLUME = {279},
      YEAR = {2014},
     PAGES = {579--605},
      ISSN = {0045-7825,1879-2138},
   MRCLASS = {65N30 (76A10)},
  MRNUMBER = {3253483},
       DOI = {10.1016/j.cma.2014.07.006},
       URL = {https://doi.org/10.1016/j.cma.2014.07.006},
}

@article{DieningKreuzerSuli2013,
  author  = {Diening, Lars and Kreuzer, Christian and S\"uli, Endre},
  title   = {Finite Element Approximation of Steady Flows of Incompressible Fluids with Implicit Power-Law-Like Rheology},
  journal = {SIAM J. Numer. Anal.},
  volume  = {51},
  number  = {2},
  pages   = {984--1015},
  year    = {2013},
  doi     = {10.1137/120873133},
}

@article {GuaschCodina2013,
    AUTHOR = {Guasch, Oriol and Codina, Ramon},
     TITLE = {Statistical behavior of the orthogonal subgrid scale
              stabilization terms in the finite element large eddy
              simulation of turbulent flows},
   JOURNAL = {Comput. Methods Appl. Mech. Engrg.},
  FJOURNAL = {Comput. Methods Appl. Mech. Engrg.},
    VOLUME = {261/262},
      YEAR = {2013},
     PAGES = {154--166},
      ISSN = {0045-7825,1879-2138},
   MRCLASS = {76F65 (65M60 76M10)},
  MRNUMBER = {3069874},
       DOI = {10.1016/j.cma.2013.04.006},
       URL = {https://doi-org.ezproxy.usach.cl/10.1016/j.cma.2013.04.006}
}

@article {Badia2010,
    AUTHOR = {Badia, Santiago and Codina, Ramon and Guti\'errez-Santacreu,
              Juan Vicente},
     TITLE = {Long-term stability estimates and existence of a global
              attractor in a finite element approximation of the
              {N}avier-{S}tokes equations with numerical subgrid scale
              modeling},
   JOURNAL = {SIAM J. Numer. Anal.},
  FJOURNAL = {SIAM J. Numer. Anal.},
    VOLUME = {48},
      YEAR = {2010},
    NUMBER = {3},
     PAGES = {1013--1037},
      ISSN = {0036-1429,1095-7170},
   MRCLASS = {65N30 (35B35 35B41 35Q30 76D05 76M10)},
  MRNUMBER = {2669399},
MRREVIEWER = {Aziz\ Belmiloudi},
       DOI = {10.1137/090766681},
       URL = {https://doi-org.ezproxy.usach.cl/10.1137/090766681},
}

@article {BadiaCodina2009,
    AUTHOR = {Badia, Santiago and Codina, Ramon},
     TITLE = {On a multiscale approach to the transient {S}tokes problem:
              dynamic subscales and anisotropic space-time discretization},
   JOURNAL = {Appl. Math. Comput.},
  FJOURNAL = {Applied Mathematics and Computation},
    VOLUME = {207},
      YEAR = {2009},
    NUMBER = {2},
     PAGES = {415--433},
      ISSN = {0096-3003,1873-5649},
   MRCLASS = {35Q35 (65M99 76D07 76M10)},
  MRNUMBER = {2489114},
       DOI = {10.1016/j.amc.2008.10.059}
}

@article {Codina2008,
    AUTHOR = {Codina, Ramon},
     TITLE = {Analysis of a stabilized finite element approximation of the
              {O}seen equations using orthogonal subscales},
   JOURNAL = {Appl. Numer. Math.},
  FJOURNAL = {Applied Numerical Mathematics. An IMACS Journal},
    VOLUME = {58},
      YEAR = {2008},
    NUMBER = {3},
     PAGES = {264--283},
      ISSN = {0168-9274,1873-5460},
   MRCLASS = {76M10 (65N30 76D05)},
  MRNUMBER = {2391446},
MRREVIEWER = {Georgios\ C.\ Georgiou},
       DOI = {10.1016/j.apnum.2006.11.011},
       URL = {https://doi-org.ezproxy.usach.cl/10.1016/j.apnum.2006.11.011},
}

@article{ErvinMiles2003,
    author = {Ervin, V. and Miles, W. },
    title = {Approximation of time-dependent viscoelastic fluid flow: {SUPG} approximation},
    journal = {SIAM J. Numer. Anal.},
    year = {2003}
}

@article{Hughes1998,
title = {The variational multiscale method—a paradigm for computational mechanics},
journal = {Comput. Methods Appl. Mech. Engrg.},
volume = {166},
number = {1},
pages = {3-24},
year = {1998},
note = {Advances in Stabilized Methods in Computational Mechanics},
issn = {0045-7825},
doi = {https://doi.org/10.1016/S0045-7825(98)00079-6},
url = {https://www.sciencedirect.com/science/article/pii/S0045782598000796},
author = {Thomas J.R. Hughes and Gonzalo R. Feijóo and Luca Mazzei and Jean-Baptiste Quincy}
}

@article {Codina1997,
    AUTHOR = {Codina, Ramon and Blasco, Jordi},
     TITLE = {A finite element formulation for the {S}tokes problem allowing
              equal velocity-pressure interpolation},
   JOURNAL = {Comput. Methods Appl. Mech. Engrg.},
  FJOURNAL = {Comput. Methods Appl. Mech. Engrg.},
    VOLUME = {143},
      YEAR = {1997},
    NUMBER = {3-4},
     PAGES = {373--391},
      ISSN = {0045-7825,1879-2138},
   MRCLASS = {76M10 (65M60 76D07)},
  MRNUMBER = {1445157},
MRREVIEWER = {Ad\'elia\ Sequeira},
       DOI = {10.1016/S0045-7825(96)01154-1},
       URL = {https://doi-org.ezproxy.usach.cl/10.1016/S0045-7825(96)01154-1}
}

@article{Hughes1995,
title = {Multiscale phenomena: Green's functions, the {D}irichlet-to-{N}eumann formulation, subgrid scale models, bubbles and the origins of stabilized methods},
journal = {Comput. Methods Appl. Mech. Engrg.},
volume = {127},
number = {1},
pages = {387-401},
year = {1995},
issn = {0045-7825},
doi = {https://doi.org/10.1016/0045-7825(95)00844-9},
url = {https://www.sciencedirect.com/science/article/pii/0045782595008449},
author = {Thomas J.R. Hughes}
}

@article {Crouzeix1987,
    AUTHOR = {Crouzeix, M. and Thom\'ee, V.},
     TITLE = {The stability in {$L_p$} and {$W^1_p$} of the
              {$L_2$}-projection onto finite element function spaces},
   JOURNAL = {Math. Comp.},
  FJOURNAL = {Mathematics of Computation},
    VOLUME = {48},
      YEAR = {1987},
    NUMBER = {178},
     PAGES = {521--532},
      ISSN = {0025-5718,1088-6842},
   MRCLASS = {41A15 (41A35 65N10 65N30)},
  MRNUMBER = {878688},
MRREVIEWER = {Dietrich\ Braess},
       DOI = {10.2307/2007825},
       URL = {https://doi.org/10.2307/2007825},
}

@article{BrooksHughes1982,
  author  = {Brooks, A. N. and Hughes, T. J. R.},
  title   = {Streamline Upwind/{P}etrov-{G}alerkin Formulations for Convection-Dominated Flows with Particular Emphasis on the Incompressible {N}avier--{S}tokes Equations},
  journal = {Comput. Methods Appl. Mech. Engrg.},
  volume  = {32},
  pages   = {199--259},
  year    = {1982},
  doi     = {10.1016/0045-7825(82)90071-8}
}

@phdthesis{CarreauYasuda,
    author = {Yasuda, Kenji},
    title = {Investigation of the analogies between viscometric and linear viscoelastic properties of polystyrene fluids.},
    school = {Massachusetts Institute of Technology, Department of Chemical Engineering},
    year = {1979}
}

@book{Teman1984,
    author = {R. Temam},
    title = {Navier–Stokes Equations. Theory and Numerical Analysis},
    publisher = {North-Holland Publishing Co.},
    year = {1984}
}

@article{Carreau1972,
  author  = {Carreau, Pierre J.},
  title   = {Rheological Equations from Molecular Network Theories},
  journal = {Trans. Soc. Rheol.},
  volume  = {16},
  number  = {1},
  pages   = {99--127},
  year    = {1972},
  doi     = {10.1122/1.549276}
}

@article{Cross1965,
  author  = {Cross, Malcolm M.},
  title   = {Rheology of Non-Newtonian Fluids: A New Flow Equation for Pseudoplastic Systems},
  journal = {J. Colloid Sci.},
  volume  = {20},
  pages   = {417--437},
  year    = {1965},
  doi     = {10.1016/0095-8522(65)90022-X}
}

@article{Quemada1977,
  author  = {Quemada, Daniel},
  title   = {Rheology of Concentrated Disperse Systems and Minimum Energy Dissipation Principle},
  journal = {Rheol. Acta.},
  volume  = {16},
  pages   = {82--94},
  year    = {1977},
  doi     = {10.1007/BF01516932}
}

@article{Eyring1936,
  author  = {Eyring, Henry},
  title   = {Viscosity, Plasticity, and Diffusion as Examples of Absolute Reaction Rates},
  journal = {J. Chem. Phys.},
  volume  = {4},
  number  = {4},
  pages   = {283--291},
  year    = {1936},
  doi     = {10.1063/1.1749836}
}

@article{BOCHEV20041471,
title = {On inf–sup stabilized finite element methods for transient problems},
journal = {Comput. Methods Appl. Mech. Engrg.},
volume = {193},
number = {15},
pages = {1471-1489},
year = {2004},
note = {Recent Advances in Stabilized and Multiscale Finite Element Methods},
issn = {0045-7825},
doi = {https://doi.org/10.1016/j.cma.2003.12.034},
url = {https://www.sciencedirect.com/science/article/pii/S0045782504000453},
author = {Pavel B. Bochev and Max D. Gunzburger and John N. Shadid}
}

@article{BARRENECHEA2007219,
title = {Pressure stabilization of finite element approximations of time-dependent incompressible flow problems},
journal = {Comput. Methods Appl. Mech. Engrg.},
volume = {197},
number = {1},
pages = {219-231},
year = {2007},
issn = {0045-7825},
doi = {https://doi.org/10.1016/j.cma.2007.07.027},
url = {https://www.sciencedirect.com/science/article/pii/S0045782507003234},
author = {Gabriel R. Barrenechea and Jordi Blasco}
}

@book{Ciarlet-2025,
author = {Ciarlet, Philippe G.},
title = {Linear and Nonlinear Functional Analysis with Applications, Second Edition},
publisher = {Society for Industrial and Applied Mathematics},
year = {2025},
doi = {10.1137/1.9781611978247},
address = {Philadelphia, PA},
edition   = {},
URL = {https://epubs.siam.org/doi/abs/10.1137/1.9781611978247},
eprint = {https://epubs.siam.org/doi/pdf/10.1137/1.9781611978247}
}

@book{BGHRR-2024,
author = {Bernardi, Christine and Girault, Vivette and Hecht, Frédéric and Raviart, Pierre-Arnaud and Rivière, Beatrice},
title = {Mathematics and Finite Element Discretizations of Incompressible Navier—Stokes Flows},
publisher = {Society for Industrial and Applied Mathematics},
year = {2024},
doi = {10.1137/1.9781611978124},
address = {Philadelphia, PA},
edition   = {},
URL = {https://epubs.siam.org/doi/abs/10.1137/1.9781611978124},
eprint = {https://epubs.siam.org/doi/pdf/10.1137/1.9781611978124}
}

@book{EG21-I,
	author = {Ern, Alexandre and Guermond, Jean-Luc},
	publisher = {Springer},
	title = {Finite elements {I}---{A}pproximation and interpolation},
	year = {2021}}

@Article{codina9,
  author =       {R. Codina},
  title =        {Stabilized finite element approximation of
                  transient incompressible flows using orthogonal subscales},
  journal =      {Comput. Methods Appl. Mech. Engrg.},
  year =         {2002},
  volume =       {191},
  pages =        {4295--4321},
}

@Article{codinap1,
  author =       {R. Codina and J. Principe and O. Guasch and S. Badia},
  title =        {Time dependent subscales in the stabilized finite element approximation of incompressible flow problems},
  journal =      {Comput. Methods Appl. Mech. Engrg.},
  year =         {2007},
  volume =       {196},
  pages =        {2413--2430},
  OPTnote = 	 {}
}

@Article{cowenc2,
  author =       {H. Coppola-Owen and R. Codina},
  title =        {A finite element model for free surface flows on fixed meshes},
  journal =      {Int. J. Numer. Methods Fluids},
  year =         {2007},
  volume =       {54},
  pages =        {1151-1171},
}

@Article{Bulicek2023,
author={Bul{\'i}{\v{c}}ek, Miroslav and M{\'a}lek, Josef and Maringov{\'a}, Erika},
title={On Unsteady Internal Flows of Incompressible Fluids Characterized by Implicit Constitutive Equations in the Bulk and on the Boundary},
journal={Journal of Mathematical Fluid Mechanics},
year={2023},
month={Jul},
day={31},
volume={25},
number={3},
pages={72},
issn={1422-6952},
doi={10.1007/s00021-023-00803-w},
url={https://doi.org/10.1007/s00021-023-00803-w}
}

@book{GaldiRannacherRobertsonTurek2008,
  author    = {Galdi, Giovanni P. and Rannacher, Rolf and Robertson, Anne M. and Turek, Stefan},
  title     = {Hemodynamical Flows: Modeling, Analysis and Simulation},
  series    = {Oberwolfach Seminars},
  volume    = {37},
  publisher = {Birkh\"auser},
  address   = {Basel},
  year      = {2008}
}

@book{BirdArmstrongHassager1987,
  author    = {Bird, R. Byron and Armstrong, Robert C. and Hassager, Ole},
  title     = {Dynamics of Polymeric Liquids, Volume 1: Fluid Mechanics},
  edition   = {2nd},
  publisher = {Wiley-Interscience},
  address   = {New York},
  year      = {1987}
}

@book{Larson1999,
  author    = {Larson, Ronald G.},
  title     = {The Structure and Rheology of Complex Fluids},
  publisher = {Oxford University Press},
  address   = {New York},
  year      = {1999}
}
\end{document}